\documentclass[11pt]{amsart}
\usepackage{amsmath, amssymb, amsbsy, amsfonts, amsthm, latexsym, amsopn, amstext, amsxtra, euscript, amscd, color, mathrsfs}
\usepackage[normalem]{ulem}
\usepackage{soul}
\usepackage{cite}
\usepackage[margin=2.5cm]{geometry}

\PassOptionsToPackage{hyphens}{url}\usepackage{hyperref}
 
\usepackage[capbesideposition=outside,capbesidesep=quad]{floatrow}

\restylefloat{table}
\usepackage{longtable}           
\usepackage{multirow,caption}    
\usepackage{amscd}
\usepackage{color,enumerate}

\newcommand{\RNum}[1]{\lowercase\expandafter{\romannumeral #1\relax}}

\usepackage[colorinlistoftodos,prependcaption,textsize=tiny]{todonotes}

\newtheorem{thm}{Theorem}[section]

\newtheorem{cor}[thm]{Corollary}
\newtheorem{prop}[thm]{Proposition}

\newtheorem{rem}[thm]{Remark}

\newtheorem{thm-con}[thm]{Theorem-Conjecture}
\numberwithin{equation}{section}

\theoremstyle{definition}
\newtheorem{defn}[thm]{Definition}
\newcommand{\prim}{a}

\newcommand{\F}{\mathbb F}

\usepackage{lipsum}
\usepackage{marginnote}

\begin{document}
\title[On Dillon's  APN hexanomials]{On the Classification of Dillon's APN Hexanomials}
\author{Daniele Bartoli}
\address{Department of Mathematics and Computer Science, University of Perugia,  06123 Perugia, Italy}
\email{daniele.bartoli@unipg.it}
\author{Giovanni Giuseppe Grimaldi}
\address{Department of Mathematics and Computer Science, University of Perugia,  06123 Perugia, Italy}
\email{giovannigiuseppe.grimaldi@unipg.it}
\author{Pantelimon St\u anic\u a}
\address{Applied Mathematics Department, Naval Postgraduate School,
Monterey, CA 93943, USA}
\email{pstanica@nps.edu}

\date{\today}


\begin{abstract}
In this paper, we undertake a systematic analysis of a class of hexanomial functions over finite fields of characteristic 2 proposed by Dillon in 2006 as potential candidates for almost perfect nonlinear (APN) functions, pushing the analysis a lot further than what has been done via the partial APN concept in Budaghyan, Kaleyski, Riera, and St\u anic\u a (Des. Codes Cryptogr. 88, 2020),  and Chase and Lison\v ek (Cryptogr. Commun. 13,  2021). These functions, defined over $\mathbb{F}_{q^2}$ where $q=2^n$, have the form
$$F(x) = x(Ax^2 + Bx^q + Cx^{2q}) + x^2(Dx^q + Ex^{2q}) + x^{3q}.$$
Using algebraic number theory and methods on algebraic varieties over finite fields, we establish necessary conditions on the coefficients $A, B, C, D, E$ that must hold for the corresponding function to be APN. Our main contribution is a comprehensive case-by-case analysis that systematically excludes large classes of Dillon's hexanomials from being APN based on the vanishing patterns of certain key polynomials in the coefficients. Through a combination of number theory, algebraic-geometric techniques and computational verification, we identify specific algebraic obstructions—including the existence of absolutely irreducible components in associated varieties and degree incompatibilities in polynomial factorizations—that prevent these functions from achieving optimal differential uniformity. Our results significantly narrow the search space for new APN functions within this family and provide a theoretical roadmap applicable to other classes of potential APN functions. We complement our theoretical work with extensive computations. Through exhaustive searches on $\mathbb{F}_{2^2}$ and $\mathbb{F}_{2^4}$ and random sampling on $\mathbb{F}_{2^6}$ and $\mathbb{F}_{2^8}$, we identified hundreds of APN hexanomials. Our classification, based on complete CCZ-equivalence testing, reveals that while many instances exist, they belong to very few distinct classes. For $q \in \{2,4\}$, all discovered APN functions are CCZ-equivalent to the known Budaghyan-Carlet (BC) family~(Budaghyan-Carlet, IEEE Trans. Inf. Th., 2008).  However, for larger dimensions, none of them seem to be equivalent to the BC family.
\end{abstract}

\maketitle

\section{Introduction and motivation}

Let $q=2^m,$ $m \in \mathbb{N},$ and denote by $\mathbb{F}_q$ the finite field with $q$ elements. For any positive integer $n$, we denote by $\F_q[X_1, \ldots, X_n]$, the ring of polynomials in $n$ indeterminates over finite field $\F_q$.

The security of a block cipher depends upon the immunity of its substitution boxes against many cryptographic attacks. For example, a low differential uniformity~\cite{KN93} is needed in order to resist the differential attacks~\cite{BS91}. For a  positive integer $n>0$, the differential uniformity of an $(n,n)$-function $F:\F_{2^n}\to\F_{2^n}$ is defined as the maximum number of solutions $x\in \mathbb{F}_{2^n}$ of the differential equation $F(x+a)+F(x)=b$, where $a\neq 0,b\in\F_{2^n}$.   The lowest possible differential uniformity of functions over finite fields of even characteristic is $2$ and such functions are called almost perfect nonlinear (APN).

Almost perfect nonlinear  functions play a fundamental role in cryptography, particularly in the design of block ciphers and stream ciphers where they provide optimal resistance against differential cryptanalysis. These functions, defined over finite fields of characteristic 2, are characterized by the property that each nonzero derivative takes each value at most twice. The search for new families of APN functions and the classification of existing ones remains one of the most active areas of research in finite field theory and cryptography.

In 2006, Dillon~\cite{Di06} suggested investigating a specific class of hexanomials (polynomials with six terms) as potential candidates for APN functions. These functions, defined over $\mathbb{F}_{q^2}$ where $q = 2^n$, have the form:
\begin{equation}
F(x) = x(Ax^2 + Bx^q + Cx^{2q}) + x^2(Dx^q + Ex^{2q}) + Gx^{3q}.
\end{equation}

The appeal of Dillon's proposal lies in the rich algebraic structure of these hexanomials, which generalizes several known constructions while potentially harboring new families of APN functions. Indeed, Budaghyan and Carlet~\cite{BC08} constructed an infinite family of APN functions of this type in 2008, demonstrating that Dillon's intuition was well-founded. However, despite this early success and subsequent investigations by various authors~\cite{BKRS20}, no systematic analysis of the entire class had been undertaken prior to this work.

The primary challenge in studying APN functions lies in the complexity of the defining condition: a function $f$ is APN if and only if for each nonzero $a \in \mathbb{F}_{q^2}$, the equation $f(x+a) + f(x) = f(y+a) + f(y)$ has only the trivial solutions $x = y$ or $x = y + a$. For Dillon's hexanomials, this condition translates into a highly nonlinear system of polynomial equations whose solutions determine whether the function achieves the desired cryptographic properties.

Our approach transforms this problem into the study of algebraic varieties over finite fields. Janwa and Wilson were the first to employ algebraic varieties to study APN functions, see \cite{Janwa}.  For a deeper introduction
to algebraic varieties we refer the interested reader to \cite{Ha77}. By reformulating the APN condition as a question about the existence of certain algebraic varieties and their irreducible components, we can apply powerful tools from algebraic geometry to obtain results that would be difficult or impossible to achieve through direct computational methods alone. This geometric perspective not only provides theoretical insights but also leads to practical algorithms for determining when specific instances of Dillon's hexanomials fail to be APN.

The main contribution of this paper is a comprehensive analysis that systematically excludes large classes of Dillon's hexanomials from being APN; see Theorem \ref{thm:summary6.11}. Through our algebraic-geometric approach, we establish necessary conditions on the coefficients $A, B, C, D, E$ that must hold for the corresponding function to have any chance of being APN. Our results significantly narrow the search space for new APN functions within this family and provide a theoretical landscape that may be applicable to other classes of potential APN functions.

The organization of our investigation follows a case-by-case analysis based on the vanishing patterns of certain key polynomials in the coefficients. We begin with the simpler case where $B = 0$, which allows us to establish our main techniques, before proceeding to the more complex general case where $B \neq 0$. Throughout, we maintain a focus on constructive proofs that not only establish non-APN behavior but also identify the specific algebraic obstructions that prevent these functions from achieving optimal differential uniformity.

To complement our theoretical analysis, we conducted extensive computational searches for APN functions within this family for several small field sizes. Using the SageMath code detailed in~\cite{GithubPS25}, we performed exhaustive searches over $\mathbb{F}_{2^2}$ and $\mathbb{F}_{2^4}$. For larger fields, namely $\mathbb{F}_{2^6}$ and $\mathbb{F}_{2^8}$, we performed large-scale random sampling of the coefficient space. The results, summarized in Appendix A, confirm that APN instances,  do exist beyond the known Budaghyan-Carlet family. Our classification, based on CCZ-invariants, reveals a rich structure of inequivalent APN functions, underscoring the significance of this class.

\section{A key theorem}
\label{sec2}

The aim in our paper is to determine the polynomials of the type
$$f_{A,B,C,D,E}(x):= x(Ax^2+Bx^q+Cx^{2q})+x^2(D x^q+E x^{2q})+x^{3q}\in \mathbb{F}_{q^2}[x] $$
which are APN (or APN permutations), or have no chance of being APN.

As usual, $f_{A,B,C,D,E}(x)$ is APN if and only if the unique solutions of 
$$f_{A,B,C,D,E}(x+a)+f_{A,B,C,D,E}(x)=f_{A,B,C,D,E}(y+a)+f_{A,B,C,D,E}(y)$$
are only $a=0$, $x=y$, or $x=y+a$. 

The equation above reads as
$$(Aa+a^{2q}E+a^qD)(x+y)^2 + (a^2A + a^{2q}C + a^qB)(x+y) + (a^2E + 
    aC + a^q)(x+y)^{2q} + (a^2D + aB + a^{2q})(x+y)^q=0.$$

Via $(x,y)\mapsto(x+y,y)$, we conclude that $f_{A,B,C,D,E}(x)$ is APN if and only if 
$$(Aa+a^{2q}E+a^qD)x^2 + (a^2A + a^{2q}C + a^qB)x + (a^2E + 
    aC + a^q)x^{2q} + (a^2D + aB + a^{2q})x^q=0$$
    has only solutions $a=0$, $x=0$, or $x=a$. 

Our first goal is to provide instances of $A,B,C,D,E\in \mathbb{F}_{q^2}$ for which the above equation has solutions beyond the trivial ones.

To this end we consider the following system

$$\begin{cases}
(Aa+a^{2q}E+a^qD)x^2 + (a^2A + a^{2q}C + a^qB)x\\
\hspace{1 cm}+ (a^2E +  aC + a^q)x^{2q} + (a^2D + aB + a^{2q})x^q=0\\
    (A^qa^q+a^{2}E^q+aD^q)x^{2q} + (a^{2q}A^q + a^{2}C^q + aB^q)x^q\\  \hspace{1 cm} + (a^{2q}E^q + a^qC^q + a)x^{2} + (a^{2q}D^q + a^qB^q + a^{2})x=0.
\end{cases}
$$
In order to prove that $f_{A,B,C,D,E}(x)$ is not APN, we need to show the existence of at least a pair $(a,x)\in \mathbb{F}_{q^2}^2$, $xa(x+a)\neq 0$, satisfying the above equations. Rewrite $a=Z_0+iZ_1$, $x=X_0+iX_1$, where $\{1,i\}$ is an $\mathbb{F}_q$-basis of $\mathbb{F}_{q^2}$. The two equations above, in terms of the variables $Z_0,Z_1,X_0,X_1$, define a variety $\mathcal{V}$ in $\mathbb{A}^4(\mathbb{F}_{q^2})$ which is $\mathbb{F}_q$-rational (i.e. the ideal generated by the two equations is fixed by the Frobenius $\varphi_q$). Consider the following change of variables $\psi$ defined by 
$$(X_0+iX_1,X_0+i^qX_1,Z_0+iZ_1,Z_0+i^qZ_1)\mapsto(X_0,X_1,Z_0,Z_1).$$
It defines an $\mathbb{F}_{q^2}$-affine equivalence between $\mathcal{V}$ and a variety $\mathcal{W}$ in $\mathbb{A}^4(\mathbb{F}_{q^2})$  defined by 
$$\begin{cases}
(AZ_0+Z_1^2 E+Z_1 D)X_0^2 + (Z_0^2A + Z_1^2C + Z_1B)X_0\\
\hspace{1 cm}+ (Z_0^2E + 
    Z_0C + Z_1)X_1^2 + (Z_0^2D + Z_0B + Z_1^2)X_1=0\\
    (A^qZ_1+Z_0^{2}E^q+Z_0D^q)X_1^2 + (Z_1^2A^q + Z_0^{2}C^q + Z_0B^q)X_1\\ 
    \hspace{1 cm}+ (Z_1^2E^q + Z_1C^q + Z_0)X_0^2 + (Z_1^2D^q + Z_1 B^q + Z_0^{2})X_0=0.
\end{cases}
$$
Notably, there is a correspondence between absolutely irreducible components of  $\mathcal{V}$ and those of $\mathcal{W}$. Also, absolutely irreducible components of $\mathcal{V}$ fixed by the Frobenius (i.e. $\mathbb{F}_q$-rational) correspond to absolutely irreducible components of $\mathcal{V}$ fixed by 
$$\phi (A,B,C,D,E,X_0,X_1,Z_0,Z_1)=(A^q,B^q,C^q,D^q,E^q,X_1,X_0,Z_1,Z_0).$$

Recall that the APN condition requires that for $a \neq 0$, the only solutions to the derivative equation are $x=0$ and $x=a$. In terms of our coordinates $X_0, X_1, Z_0, Z_1$, these trivial solutions correspond to specific hyperplanes in the affine space $\mathbb{A}^4(\mathbb{F}_{q^2})$. Additionally, degenerate cases where $a$ is not a valid parameter correspond to $Z_0=0$ or $Z_1=0$. We formalize this locus as follows.

\begin{defn}[Forbidden planes]
\label{def:forbidden_locus}
The \textit{forbidden locus}, denoted by $\Pi$, is the union of the six hyperplanes $\pi_1, \dots, \pi_6$ in $\mathbb{A}^4(\mathbb{F}_{q^2})$ defined by the trivial APN solutions and the degenerate parameters:
\begin{align*}
    \pi_1 &: X_0 = 0, & \pi_2 &: X_1 = 0, \\
    \pi_3 &: X_0 + Z_0 = 0, & \pi_4 &: X_1 + Z_1 = 0, \\
    \pi_5 &: Z_0 = 0, & \pi_6 &: Z_1 = 0.
\end{align*}
Consequently, $f_{A,B,C,D,E}$ is APN if and only if the set of $\mathbb{F}_q$-rational points of the variety $\mathcal{W}$ is contained in $\Pi$.
\end{defn}

We recall a refinement of the classical Lang-Weil bound, which will be crucial for proving a non-existence result for sufficiently large $q$.
\begin{thm}\textup{\cite{MR2206396}}
\label{th:cafmat}
Let $\mathcal{V}\subset \mathbb{A}^N(\mathbb{F}_q)$ be an $\mathbb{F}_q$-irreducible variety of dimension $r$ and degree $d$. If $q>2(r+1)d^2$ then
$$
|\#\mathcal{V}(\mathbb{F}_q)-q^r|\le (d-1)(d-2)q^{r-1/2}+5d^{\frac{13}{3}}q^{r-1}.
$$
\end{thm}

The following is the key result of this paper. 

\begin{thm}\label{Th:Key}
    Suppose that there exists a variety $\mathcal{C}$ contained in $\mathcal{W}$ that is absolutely irreducible and fixed by 
    $\phi$, where $\phi(A,B,C,D,E,X_0,X_1,Z_0,Z_1)=(A^q,B^q,C^q,D^q,E^q,X_1,X_0,Z_1,Z_0)$  and not contained in the  hyperplanes $X_0=0$, $X_1=0$, $Z_0=X_0$, $Z_1=X_1$, $Z_0=0$, $Z_1=0$.
    Then, if $q$ is large enough $f_{A,B,C,D,E}(x)$ is not APN.
    In particular, if the dimension of $\mathcal{C}$ is $2$ and $q \geq 2^{20}$,  then $f_{A,B,C,D,E}(x)$ is not APN.
    Conversely, if $\mathcal{W}$ is contained in $\Pi$, then $f_{A,B,C,D,E}(x)$ is  APN.
\end{thm}
\begin{proof}
We reformulate the APN condition in geometric terms and apply the Lang-Weil theorem to count rational points on the associated variety.
Recall that $f_{A,B,C,D,E}$ is APN if and only if for all nonzero $a \in \mathbb{F}_{q^2}$, the equation
$$f_{A,B,C,D,E}(x+a)+f_{A,B,C,D,E}(x)=f_{A,B,C,D,E}(y+a)+f_{A,B,C,D,E}(y)$$
has only the trivial solutions $x=y$ or $x=y+a$.

Via the change of variables $(x,y) \mapsto (x+y,y)$, this is equivalent to requiring that for all nonzero $a \in \mathbb{F}_{q^2}$, the equation
\begin{equation}\label{eq:apn_condition}
(Aa+a^{2q}E+a^qD)x^2 + (a^2A + a^{2q}C + a^qB)x + (a^2E + aC + a^q)x^{2q} + (a^2D + aB + a^{2q})x^q=0
\end{equation}
has only the solutions $a=0$, $x=0$, or $x=a$.

Write $a=Z_0+iZ_1$ and $x=X_0+iX_1$ where $\{1,i\}$ is an $\mathbb{F}_q$-basis of $\mathbb{F}_{q^2}$ with $i^q = i + c$ for some $c \in \mathbb{F}_q$ (depending on the choice of basis). The APN condition translates to: for all $(Z_0,Z_1) \in \mathbb{F}_q^2 \setminus \{(0,0)\}$, Equation~\eqref{eq:apn_condition} has only solutions $(X_0,X_1) \in \{(0,0), (Z_0,Z_1)\}$.

Consider the system obtained by taking Equation~\eqref{eq:apn_condition} together with its $q$-th power (Frobenius conjugate):
$$\begin{cases}
(Aa+a^{2q}E+a^qD)x^2 + (a^2A + a^{2q}C + a^qB)x\\
\hspace{1 cm}+ (a^2E + aC + a^q)x^{2q} + (a^2D + aB + a^{2q})x^q=0\\
(A^qa^q+a^{2}E^q+aD^q)x^{2q} + (a^{2q}A^q + a^{2}C^q + aB^q)x^q\\
\hspace{1 cm} + (a^{2q}E^q + a^qC^q + a)x^{2} + (a^{2q}D^q + a^qB^q + a^{2})x=0.
\end{cases}$$

Expressing $a$ and $x$ in terms of the basis, this system defines an $\mathbb{F}_q$-rational variety $\mathcal{V}$ in $\mathbb{A}^4(\mathbb{F}_{q^2})$.

Under the change of variables $\psi$ defined by 
$$(X_0+iX_1,X_0+i^qX_1,Z_0+iZ_1,Z_0+i^qZ_1)\mapsto(X_0,X_1,Z_0,Z_1),$$
the variety $\mathcal{V}$ is mapped to the variety $\mathcal{W}$ defined by:
$$\begin{cases}
F_1(X_0,X_1,Z_0,Z_1):=&(AZ_0+Z_1^2 E+Z_1 D)X_0^2 + (Z_0^2A + Z_1^2C + Z_1B)X_0\\
&+ (Z_0^2E + Z_0C + Z_1)X_1^2 + (Z_0^2D + Z_0B + Z_1^2)X_1=0\\
F_2(X_0,X_1,Z_0,Z_1):=&(A^qZ_1+Z_0^{2}E^q+Z_0D^q)X_1^2 + (Z_1^2A^q + Z_0^{2}C^q + Z_0B^q)X_1\\ 
&+ (Z_1^2E^q + Z_1C^q + Z_0)X_0^2 + (Z_1^2D^q + Z_1 B^q + Z_0^{2})X_0=0.
\end{cases}$$

The change of variables $\psi$ is an $\mathbb{F}_{q^2}$-isomorphism, and there is a bijection between absolutely irreducible components of $\mathcal{V}$ and those of $\mathcal{W}$.
 
 The APN condition now reads as follows: $f_{A,B,C,D,E}(x)$ is APN if and only if every $\mathbb{F}_q$-rational point of $\mathcal{W}$ lies in the union of the hyperplanes $\pi_i$, $i=1,\ldots,6$.

Moreover, an absolutely irreducible component of $\mathcal{V}$ is $\mathbb{F}_q$-rational (fixed by the $q$-th power Frobenius) if and only if the corresponding component of $\mathcal{W}$ is fixed by the morphism
$$\phi(A,B,C,D,E,X_0,X_1,Z_0,Z_1)=(A^q,B^q,C^q,D^q,E^q,X_1,X_0,Z_1,Z_0).$$

Now suppose that $\mathcal{W}$ contains an absolutely irreducible variety $\mathcal{C}$ that is fixed by $\phi$. We will show that this implies $f_{A,B,C,D,E}$ is not APN for sufficiently large $q$.
The variety $\mathcal{C}$, being absolutely irreducible and contained in $\mathbb{A}^4(\mathbb{F}_{q^2})$, has dimension $r$ where $0 \leq r \leq 4$. 

If $r=0$ then $\mathcal{C}$ consists of one point and by our assumption, it is not contained in $\bigcup_{i}\pi_i$ and thus $f_{A,B,C,D,E}$ is not APN. 

Suppose that $r>0$. Then $\mathcal{C}$ consists of 
$$  q^r + O(q^{r-1/2})$$
points fixed by $\phi$, by the Lang-Weil bound. Since the intersection between $\mathcal{C}$ and each $\pi_i$, $i=1,\ldots,6$, is a variety of dimension $r-1$ and degree at most $d$, it contains at most 
$$d(q^{r-1} + O(q^{r-3/2})) $$
points fixed by $\phi$. 

Thus we conclude that $f_{A,B,C,D,E}$ is not APN if 
\begin{equation}\label{Estimate}
q^r + O(q^{r-1/2})-6d(q^{r-1} + O(q^{r-3/2}))
\end{equation}
is positive. 

In particular if $r=2$ then $d\leq 14$ (in fact $\mathcal{W}$ is the  complete intersection of two quartics in $\mathbb{A}^4$ and two components are $\pi_1$ and $\pi_2$). In order to estimate the error terms in Equation \eqref{Estimate}, we can make use of Theorem \ref{th:cafmat} and  we conclude that if $q\geq 2^{20}$ then the quantity in Equation \eqref{Estimate} is positive and $f_{A,B,C,D,E}$ is not APN.
\end{proof}

\begin{rem}
\label{rem:bound_discussion}
The lower bound $q \geq 2^{20}$ established in Theorem \ref{Th:Key} is a sufficient condition derived from the worst-case error terms in the Lang-Weil bound (using for instance the Cafure-Matera extimates). This ensures that the number of rational points on the variety $\mathcal{C}$ strictly exceeds the number of points contained in the forbidden hyperplanes. However, from a geometric perspective, the existence of an absolutely irreducible component $\mathcal{C} \not\subseteq \bigcup \pi_i$ is a structural obstruction to the APN property that typically manifests for much smaller $q$. In the vast majority of cases, such a variety will possess $\mathbb{F}_q$-rational points outside the forbidden locus long before the asymptotic lower bound is reached. This is strongly supported by our computational classification in Section \ref{sec:computational_short}, where we observe that for $q \in \{2, 4, 8, 16\}$, the function $f_{A,B,C,D,E}$ fails to be APN  when the geometric conditions for the existence of such a component are met.
\end{rem}

Our next aim is to provide conditions on the coefficients $A,B,C,D,E\in \mathbb{F}_{q^2}$ for which Theorem~\ref{Th:Key} applies. 
First, note that the coefficient of $X_1^2$ in the first equation is non-vanishing (as polynomial in the remaining variables). 
We continue our investigation by simplifying the two equations $$F_1(X_0,X_1,Z_0,Z_1)=0 \textrm{ and } F_2(X_0,X_1,Z_0,Z_1)=0$$ defining $\mathcal{W}$.
Let 
\allowdisplaybreaks
\begin{align*}
    G(X_0,X_1,Z_0,Z_1)&:=(Z_0^2 E^q + Z_0 D^q + Z_1 A^q)F_1(X_0,X_1,Z_0,Z_1)\\
    &\qquad+(Z_0^2 E + Z_0 C + Z_1)F_2(X_0,X_1,Z_0,Z_1)\\
    &=(Z_0^3 A E^q + Z_0^3 E + Z_0^2 Z_1 C^{q} E + Z_0^2 Z_1 D E^q + Z_0^2 A D^{q} + Z_0^2 C + Z_0 Z_1^2 C E^q\\
        &\qquad+ Z_0 Z_1^2 D^{q} E + Z_0 Z_1 A^{q+1} + Z_0 Z_1 C^{q+1} + Z_0Z_1 D^{q+1} + Z_0Z_1 + Z_1^3 A^{q} E\\
        &\qquad+ Z_1^3 E^q + Z_1^2 A^{q} D + Z_1^2 C^{q})X_0^2\\
        &\qquad+ (Z_0^4 A E^q + Z_0^4 E + Z_0^3 A D^{q} + Z_0^3 C + Z_0^2 Z_1^2 C E^q + Z_0^2 Z_1^2 D^{q} E\\
        &\qquad+ Z_0^2 Z_1 A^{q+1} + Z_0^2 Z_1 B E^q + Z_0^2 Z_1 B^{q} E + Z_0^2 Z_1 + Z_0 Z_1 B D^{q}\\
        &\qquad+ Z_0 Z_1 B^{q} C + Z_1^3 A^{q} C + Z_1^3 D^{q} + Z_1^2 A^{q} B + Z_1^2 B^{q})X_0\\
        &\qquad +(Z_0^4 C^{q} E + Z_0^4 D E^q + Z_0^3 B E^q + Z_0^3 B^{q} E + Z_0^3  C^{q+1} + Z_0^3D^{q+1}\\
        &\qquad+Z_0^2 Z_1^2 A^{q} E + Z_0^2 Z_1^2 E^q + Z_0^2 Z_1 A^{q} D + Z_0^2 Z_1 C^{q} + Z_0^2 B D^{q}\\
        &\qquad+Z_0^2 B^{q} C + Z_0 Z_1^2 A^{q} C + Z_0 Z_1^2 D^{q} + Z_0 Z_1 A^{q} B + Z_0 Z_1 B^{q})X_1.
\end{align*}

\begin{prop}\label{Prop:CoeffX_1}
    If the coefficient of $X_1$ in $G(X_0,X_1,Z_0,Z_1)$ vanishes, then one of the following holds:
\begin{enumerate}
    \item[$(C1)$] $A\neq 0$, $C=D=0$, $A^qB=B^q$, $A^qE=E^q$; or 
    \item[$(C2)$] $ACD\neq 0$, $A^{q+1}=1$, $D=AC^q$, $B^q=A^qB$, $E^q=A^qE$.
\end{enumerate}
\end{prop}
 \begin{proof}
    The coefficient of $X_1$ in $G(X_0,X_1,Z_0,Z_1)$ vanishes if and only if the following system holds:
    \begin{align}
    C^{q} E + D E^q &= 0 \tag{i}\\
    B E^q + B^{q} E + C^{q+1} + D^{q+1} &= 0 \tag{ii}\\
    A^{q} E + E^q &= 0 \tag{iii}\\
    A^{q} D + C^{q} &= 0 \tag{iv}\\
    B D^{q} + B^{q} C &= 0 \tag{v}\\
    A^{q} C + D^{q} &= 0 \tag{vi}\\
    A^{q} B + B^{q} &= 0. \tag{vii}
    \end{align}

Note that $A=0$ forces $B=C=D=E=0$ (trivial case), so assume $A\neq 0$.
From (iv) and (vi), either $C=D=0$ or both $C,D\neq 0$.

\textbf{Case 1: $C=D=0$ and $A\neq 0$.}

Equations (i), (iv), (v), (vi) are automatically satisfied. Equations (iii) and (vii) give
$$A^qE = E^q, \quad A^qB = B^q.$$
Equation (ii) becomes $BE^q + B^qE = 0$. 
From $A^qE=E^q$ in $\mathbb{F}_{q^2}$, applying the $q$-power Frobenius (using $E^{q^2}=E$ and $A^{q^2}=A$),
$$A^{q^2}E^q = E^{q^2} \implies AE^q = E.$$
Similarly, from $A^qB=B^q$, we get $AB^q=B$.

Now we verify Equation (ii). We have $B^qE = (A^qB)(AE^q)$ since $B^q = A^qB$ and $E = AE^q$. Therefore,
$$BE^q + B^qE = BE^q + (A^qB)(AE^q) = B(A^qE) + (A^qB)(AE^q).$$
Since $E^q = A^qE$ and $E = AE^q$, we have
$$B(A^qE) + (A^qB)(AE^q) = BE^q + (A^qB)E = B(A^qE) + A^q(BE) = A^q(BE) + A^q(BE) = 0$$
in characteristic 2.
This gives Condition $(C1)$.

\textbf{Case 2: $CD\neq 0$ and $A\neq 0$.}

From (iv): $D = C^q/A^q$. From (vi): $C = D^q/A^q = (C^q/A^q)^q/A^q = C^{q^2}/A^{q(q+1)}$. 
Since $C^{q^2} = C$ in $\mathbb{F}_{q^2}$, we have $C = C/A^{q(q+1)}$, which gives (since $C \neq 0$),
$$A^{q(q+1)} = 1.$$
Taking the $q$-th power: $A^{q^2(q+1)} = 1$, hence $A^{q+1} = 1$ (using $A^{q^2}=A$).
With $A^{q+1}=1$, from (iv), $D = C^q/A^q$. Since $A^{q+1}=1$, we have $A^q = A^{-1}$, so
$$D = C^q /A^{-1} =  AC^q.$$
From (iii) and (vii), we get $A^qE=E^q$ and $A^qB=B^q$.

Verification of remaining equations confirms consistency with these values.
This gives condition $(C2)$.
\end{proof}

\begin{prop}\label{Prop:Condition_(C1)}
    Suppose that Condition $(C1)$ holds. If $q$ is large enough, then $f_{A,B,C,D,E}(x)$ is APN if and only if $A^{q+1}+1\neq 0$.
\end{prop}
\begin{proof}
    Consider first $A^{q+1}+1\neq 0$. Then Condition $(C1)$ yields also $B=E=0$. Then 
    \begin{eqnarray*}
        G(X_0,X_1,Z_0,Z_1)&:=& (A^{q+1}+1)Z_0Z_1X_0(X_0+Z_0),\\
        F_2(X_0,X_1,Z_0,Z_1)&:=& A^q X_1^2 Z_1 + A^q X_1 Z_1^2 + X_0^2 Z_0 + X_0 Z_0^2,
    \end{eqnarray*}
and the components of $G(X_0,X_1,Z_0,Z_1)=F_2(X_0,X_1,Z_0,Z_1)=0$ are contained in the union of the hyperplanes $X_0=0$, $X_1=0$, $Z_0=X_0$, $Z_1=X_1$, $Z_0=0$, $Z_1=0$, and by Theorem \ref{Th:Key}   $f_{A,B,C,D,E}(x)$ is APN. 
    
    Suppose that $A^{q+1}+1=0$. In this case $\mathcal{W}$ collapses to a unique equation (fixed by $\phi$)
    $$H(X_0,X_1,Z_0,Z_1):= (A Z_0 + E Z_1^2)X_0^2 +(A Z_0^2 + B Z_1)X_0+X_1(B Z_0 + E X_1 Z_0^2 + X_1 Z_1 + Z_1^2)=0.$$
    \begin{itemize}
        \item $(B,E)\neq (0,0)$. Note that $B Z_0 + E X_1 Z_0^2 + X_1 Z_1 + Z_1^2$ and $A Z_0 + E Z_1^2$ are both  irreducible and of degree at most three and at least one. A putative factorization of  $H(X_0,X_1,Z_0,Z_1)$   is 
        $$((A Z_0 + E Z_1^2)X_0+L_1(X_1,Z_0,Z_1))(X_0+L_2(X_1,Z_0,Z_1)),$$
        where $L_2(X_1,Z_0,Z_1)$ is a divisor of $X_1(B Z_0 + E X_1 Z_0^2 + X_1 Z_1 + Z_1^2)$. This implies that 
        $$H(L_2,X_1,Z_0,Z_1) \equiv 0.$$
        If $\deg(L_2)>1$ this provides a clear contradiction, since $\deg((A Z_0 + E Z_1^2)L_2^2)\geq 1+2\deg(L_2)$ and $H(L_2,X_1,Z_0,Z_1) \not\equiv 0$. On the other hand if $\deg(L_2)\leq 1$ then  $L_2=\lambda X_1$ or $L_2=\lambda $ with $\lambda \in \overline{\mathbb{F}_{q}}$. In this case, by easy computations $H(L_2,X_1,Z_0,Z_1)$  does not vanish. This shows that $H(X_0,X_1,Z_0,Z_1)$ is absolutely irreducible. 
        
        \item $B=E= 0$. In this case 
        $$H(X_0,X_1,Z_0,Z_1):= A Z_0X_0^2 +A Z_0^2X_0+X_1^2 Z_1 + X_1Z_1^2.$$
        Since $H(X_0,X_1,Z_0,1)=X_1^2  + X_1 +AX_0Z_0( X_0 +Z_0)$ has constant term (in $X_1$) of degree three in $X_0$ and $Z_0$,  $H(X_0,X_1,Z_0,1)$ and thus $H(X_0,X_1,Z_0,Z_1)$ is absolutely irreducible. 
    \end{itemize}
This shows that when $A^{q+1}+1=0$, $\mathcal{W}$ is fixed by  $\phi$, absolutely irreducible and clearly not contained in the forbidden hyperplanes. Thus $f_{A,B,C,D,E}(x)$ is not APN.   
\end{proof}

\begin{prop}\label{Prop:Condition_(C2)}
    Suppose that Condition $(C2)$ holds. If $q$ is large enough, then $f_{A,B,C,D,E}(x)$ is not APN.
\end{prop}
\begin{proof}
    In this case $\mathcal{W}$ collapses to a unique equation (fixed by $\phi$)
    \begin{eqnarray*}
     H(X_0,X_1,Z_0,Z_1)&:=& (A C^q Z_1 + A Z_0 + E Z_1^2)X_0^2 +(A Z_0^2 + B Z_1 + C Z_1^2)X_0\\
     &&\qquad+X_1(A C^q Z_0^2 + B Z_0 + C X_1 Z_0 + E X_1 Z_0^2 + X_1 Z_1 + Z_1^2)=0.   
    \end{eqnarray*}
    Recall that $AC\neq 0$. 
     Note that $A C^q Z_0^2 + B Z_0 + C X_1 Z_0 + E X_1 Z_0^2 + X_1 Z_1 + Z_1^2$ and $A C^q Z_1 + A Z_0 + E Z_1^2$ are both  irreducible and of degree at most three and at least one. A putative factorization of  $H(X_0,X_1,Z_0,Z_1)$   is 
        $$((A C^q Z_1 + A Z_0 + E Z_1^2)X_0+L_1(X_1,Z_0,Z_1))(X_0+L_2(X_1,Z_0,Z_1)),$$
        where $L_2(X_1,Z_0,Z_1)$ is a divisor of $X_1(B Z_0 + E X_1 Z_0^2 + X_1 Z_1 + Z_1^2)$. This implies that 
        $$H(L_2,X_1,Z_0,Z_1) \equiv 0.$$
        If $\deg(L_2)>1$ this provides a clear contradiction, since $\deg((A Z_0 + E Z_1^2)L_2^2)\geq 1+2\deg(L_2)$ and $H(L_2,X_1,Z_0,Z_1) \not\equiv 0$. On the other hand if $\deg(L_2)\leq 1$ then  $L_2=\lambda X_1$ or $L_2=\lambda $ with $\lambda \in \overline{\mathbb{F}_{q}}$. 
        
        Now, $H(\lambda,X_1,Z_0,Z_1)\not \equiv 0$ since the coefficient of $X_1Z_1^2$ is $1$.

        Also, if $(B,E)\neq (0,0)$ then $H(\lambda X_1,X_1,Z_0,Z_1)\not \equiv 0$ since the coefficient of $Z_0^2X_1^2$ and  $X_1 Z_0$ are $E$ and $B$. 
        
        This shows that when $(B,E)\neq (0,0)$ $\mathcal{W}$ is fixed by  $\phi$, absolutely irreducible and clearly not contained in the forbidden hyperplanes. Thus $f_{A,B,C,D,E}(x)$ is not APN.

        Consider now the case $(B,E)= (0,0)$. If $H(\lambda X_1,X_1,Z_0,Z_1) \equiv 0$, then 
        $$\lambda =C^q, \quad \lambda^2 A= C, \qquad C\lambda=1, \qquad \lambda^2 A C^q=1.$$
        This yields $A=C^3$, $C^{q+1}=1$. In this case, after clearing the denominators
        $$H(X_0,X_1,Z_0,Z_1)=(C Z_0 + Z_1)(C X_0 + X_1)(C X_0 + C Z_0 + X_1 + Z_1).$$
        Each of these three factors is fixed by $\phi$ and defines a hypersurface not contained in the forbidden hyperplanes. Also in this case $f_{A,B,C,D,E}(x)$ is not APN.
\end{proof}

From now on, we suppose that neither Condition (C1) nor Condition (C2) holds. Thus  the coefficient of $X_1$ in $G(X_0,X_1,Z_0,Z_1)$ is non-vanishing and eliminating $X_1$ in $F_2(X_0,X_1,Z_0,Z_1)=0$ via $G(X_0,X_1,Z_0,Z_1)=0$ one gets 
\begin{equation}
\label{FactorizationS}
\overline{G}(X_0,Z_0,Z_1) :=(Z_0^2E^q + Z_0D^q + Z_1A^q)X_0(X_0 + Z_0)(a_2 X_0^2+a_1X_0+a_0)=0,
\end{equation}
where 
\allowdisplaybreaks
\begin{align*}
 a_2&:= (Z_0^3 A E^q + Z_0^3 E + Z_0^2 Z_1 C^q E + Z_0^2 Z_1 D E^q + Z_0^2 A D^q + Z_0^2 C + Z_0 Z_1^2 C E^q + Z_0 Z_1^2 D^q E\\ 
 &\quad+ Z_0 Z_1 A^{q+1} + Z_0 Z_1 C^{q+1} + Z_0 Z_1 D^{q+1} + Z_0 Z_1 + Z_1^3 A^q E + Z_1^3 E^q + Z_1^2 A^q D + Z_1^2 C^q  )^2\\
a_1 &:= a_2Z_0;\\
a_0 &:= (Z_0^3 C^q E + Z_0^3 D E^q + Z_0^2 B E^q + Z_0^2 B^q E + Z_0^2 C^{q+1} + Z_0^2 D^{q+1} + Z_0 Z_1^2 A^q E + Z_0 Z_1^2 E^q\\
&\quad + Z_0 Z_1 A^q D + Z_0 Z_1 C^q + Z_0 B D^q + Z_0 B^q C + Z_1^2 A^q C + Z_1^2 D^q + Z_1 A^q B + Z_1 B^q)\\
&\quad\cdot(Z_0^4 A C^q + Z_0^4 D + Z_0^3 A B^q + Z_0^3 B + Z_0^2 Z_1^2 A^{q+1} + Z_0^2 Z_1^2 C^{q+1} + Z_0^2 Z_1^2 D^{q+1} + Z_0^2 Z_1^2\\
&\quad + Z_0^2 Z_1 B C^q + Z_0^2 Z_1 B^q D + Z_0 Z_1^2 B D^q + Z_0 Z_1^2 B^q C + Z_1^4 A^q C + Z_1^4 D^q + Z_1^3 A^q B + 
        Z_1^3 B^q).
\end{align*}
Note that $Z_0^2E^q + Z_0D^q + Z_1A^q$ is a non-vanishing factor. 
Let us consider $a_0=g_1g_2$, $a_2=g_3^2$, as in the factorization above.
Let $\mathcal{Z}$ be defined by 
$$\begin{cases}
G(X_0,X_1,Z_0,Z_1)=0\\
a_2X_0^2+a_1X_0+a_0=0.
\end{cases}
$$

Clearly $\mathcal{Z}\subset \mathcal{W}$. It is possible to check that the surface $\mathcal{Z}$ is closed under the action of $\phi$. Also, if $\ell=\gcd(a_2,a_1,a_0)$, $a_2\neq 0$, then 
$$\begin{cases}
\widetilde{G}(X_0,X_1,Z_0,Z_1)=0\\
(a_2X_0^2+a_1X_0+a_0)/\ell=0
\end{cases}
$$ 
is also fixed by $\phi$.  We consider a variety $\widetilde{\mathcal{Z}}\subset \mathcal{Z}$ that is birationally equivalent to the surface $\mathcal{H}: a_2X_0^2+a_1X_0+a_0=0$. Since absolute irreducibility is preserved under birational equivalence, we may focus our analysis on $\mathcal{H}$ itself.

Thus, in order to prove the existence of a component in $\mathcal{W}$ absolutely irreducible and fixed by $\phi$, it is sufficient to prove that $a_2X_0^2+a_1X_0+a_0$ has degree $2$ in $X_0$ and the non-existence of factors in $a_2X_0^2+a_1X_0+a_0$ of degree $1$ in $X_0$.

Note that the polynomial $a_2X_0^2+a_1X_0+a_0$ contains a factor of degree $1$ in $X_0$ if and only if there exist $f,h \in \overline{\mathbb{F}_q}[Z_0,Z_1]$ such that 
     \begin{equation}\label{Equation1}
     g_3^2f^2+g_3^2Z_0fh+g_1g_2h^2=0.
     \end{equation}

As a notation,  for a polynomial $\ell \in \overline{\mathbb{F}_q}[Z_0,Z_1]$, we denote by $\ell^{(i)}$ and $\ell^{(L)}$ the homogeneous part of degree $i$ and the lowest (non-vanishing)  homogeneous part in $\ell$, respectively.

\begin{rem}\label{Remark}
    
    In what follows we will make use a number of times of the following observation. Let us consider  $a_2 X_0^2+a_1X_0+a_0$, where $a_0,a_1,a_2\in \mathbb{F}_q[Z_0,Z_1]$. If $a_2 X_0^2+a_1X_0+a_0$ is fixed by $X_0 \mapsto X_0+Z_0$ then putative degree-1 factors (in $X_0$) of $a_2 X_0^2+a_1X_0+a_0$ are of the type 
    $ \beta (X_0+Z_0)+\gamma$  and  $\beta X_0+\gamma$,
    for some $\beta,\gamma \in \overline{\mathbb{F}}_{q}[Z_0,Z_1]$ and thus if $a_2 X_0^2+a_1X_0+a_0$ contains factors of degree one in $X_0$ it must hold 
    $ a_2 X_0^2+a_1X_0+a_0=\alpha (\beta (X_0+Z_0)+\gamma)(\beta X_0+\gamma)$,
    for some $\alpha,\beta,\gamma \in \overline{\mathbb{F}}_{q}[Z_0,Z_1]$. Also, since $a_2=g_3^2$, 
$\alpha \gamma (\gamma+\beta Z_0)=a_0$, and $\alpha \beta^2=g_3^2$.
Without loss of generality, we may assume $\alpha=1$. Indeed, if $\alpha \ne 1$, then from $\alpha\beta^2 = g_3^2$, we see that $\alpha$ must be a perfect square, say $\alpha = \alpha_0^2$ for some $\alpha_0 \in \overline{\mathbb{F}}_q[Z_0,Z_1]$. We could then write $a_2X_0^2+a_1X_0+a_0 = [\alpha_0\beta X_0 + \alpha_0\gamma][\alpha_0\beta(X_0+Z_0) + \alpha_0\gamma]$, which allows us to replace $(\beta,\gamma)$ with $(\alpha_0\beta, \alpha_0\gamma)$ and reduce to the case where $\alpha=1$. 

\end{rem}

\section{Case \texorpdfstring{$B=0$}{B=0}}
We start our investigation with the case $B=0$.

\begin{thm}
\label{B0_1}
Suppose that conditions $(C1)$ and $(C2)$ do not hold.
Let $B=AC^q+D=0$ and let $q \geq 2^{20}$. Then:
\begin{enumerate}
\item If $(A^{q+1}+1)(C^{q+1}+1) \neq 0$, then $f_{A,B,C,D,E}(x)$ is not APN.

\item If $(A^{q+1}+1)=0$, $(C^{q+1}+1)\neq 0$, $AE^q+E \ne 0$, and $T^3+CT^2+AC^qT+A$ has a root $k \in \mathbb{F}_{q^2}$ with $k \ne 1$ and $k^{q+1}=1$, then $f_{A,B,C,D,E}(x)$ is not APN.

\item  If $(A^{q+1}+1)\neq 0$, $(C^{q+1}+1)=0$, 
then $f_{A,B,C,D,E}(x)$ is not APN. 

\item If $(A^{q+1}+1)=0$, $(C^{q+1}+1)=0$, and $AE^q+E \ne 0$, then $f_{A,B,C,D,E}(x)$ is  not APN. 
\end{enumerate}
\end{thm}

\begin{rem}
The APN status of the function remains open when $(A^{q+1}+1)=0$, $(C^{q+1}+1)\neq 0$, $AE^q+E \ne 0$, and $T^3+CT^2+AC^qT+A$ has no roots in $\mathbb{F}_{q^2}$, the function may be APN if the conic $AE^qZ_0^2 + CZ_0 + Z_1 = 0$ contributes only trivial solutions. This requires further analysis of the $\mathbb{F}_q$-rational points on this conic.
\end{rem}

\begin{proof}[Proof of Theorem~\textup{\ref{B0_1}}]
Recall from Section~\ref{sec2} that $\mathcal{W}$ is defined by the system
$$\begin{cases}
(AZ_0+Z_1^2 E+Z_1 D)X_0^2 + (Z_0^2A + Z_1^2C + Z_1B)X_0\\
\hspace{1 cm}+ (Z_0^2E + Z_0C + Z_1)X_1^2 + (Z_0^2D + Z_0B + Z_1^2)X_1=0\\
(A^qZ_1+Z_0^{2}E^q+Z_0D^q)X_1^2 + (Z_1^2A^q + Z_0^{2}C^q + Z_0B^q)X_1\\ 
\hspace{1 cm}+ (Z_1^2E^q + Z_1C^q + Z_0)X_0^2 + (Z_1^2D^q + Z_1 B^q + Z_0^{2})X_0=0.
\end{cases}$$

Under the hypotheses $B=0$ and $AC^q+D=0$, the factorization in \eqref{FactorizationS} reads
$$
X_0(X_0+Z_0)(A^qCZ_0 + A^qZ_1 + E^qZ_0^2)(b_2X_0^2+b_1X_0+b_0)=0,
$$
where
\begin{align*}
b_2:= & \Big(  (A^{q+1}+1)(C^{q+1}+1)Z_0Z_1 +C(A^{q+1}+1)Z_0^2+C^q(A^{q+1}+1)Z_1^2 \\
&\quad +C^q(AE^q+E)Z_0^2Z_1+C(A^qE+E^q)Z_0Z_1^2\\
&\quad +(AE^q+E)Z_0^3+(A^qE+E^q)Z_1^3 \Big)^2, \\
b_1:=& b_2Z_0, \\
b_0:=& (A^{q+1}+1)(C^{q+1}+1)Z_0^3Z_1^2\Big( (A^{q+1}+1)C^{q+1}Z_0 + (A^{q+1}+1)C^qZ_1 \\
&\quad + (AE^q+E)C^qZ_0^2 + (A^qE+ E^q)Z_1^2 \Big).
\end{align*}

We prove each part separately.

\item \textbf{Proof of Part (1):} Assume $(A^{q+1}+1)(C^{q+1}+1) \neq 0$.

First, assume $AE^q+E = 0$. Note that $C\neq 0$ otherwise we are in case $(C1)$ or in the trivial case $A=B=C=D=E=0$. Now, by Remark~\ref{Remark}, if $b_2X_0^2+b_1X_0+b_0$ splits into degree-one factors in $X_0$, then
$$b_2X_0^2+b_1X_0+b_0=(\beta X_0 + \gamma)(\beta (X_0+Z_0) + \gamma),$$
with $\beta^2=b_2$ and $\gamma(\gamma + \beta Z_0) = b_0$. Since $\deg(\beta)=2$ and $\deg(b_0)=6$, then $\gamma$ is homogeneous of degree 3. Put
$$
\gamma =r Z_0^3+sZ_0^2Z_1+tZ_0Z_1^2+uZ_1^3.
$$ Such an $\gamma$ must make $h:=\gamma^2+\gamma\beta Z_0 + b_0$ the zero polynomial in $Z_0$ and $Z_1$. Let $h:=\sum_{i=0}^6h_iZ_1^iZ_0^{6-i}$, where
\begin{eqnarray*}
    h_0&=& r \Big(r+(A^{q+1}+1)C \Big),\\
     h_1&=& (A^{q+1}+1)\Big(r(C^{q+1}+1)+sC\Big) ,\\ 
      h_2&=& rC^q(A^{q+1} + 1)+ s^2 + s(A^{q+1}+1)(C^{q+1} + 1) + tC(A^{q+1}+1) + A^{q+1}C^{q+1}(C^{q+1}+1) ,\\
      h_3&=& (A^{q+1}+1)\Big( sC^q+t(C^{q+1}+1)+uC+C^q(C^{q+1}+1) \Big), \\
     h_4&=& t^2 + tC^q(A^{q+1}+1) + u(A^{q+1}+1)(C^{q+1} + 1) ,\\
      h_5&=& C^q(A^{q+1}+1)u,\\
       h_6&=& u^2.
\end{eqnarray*}

From $h_1=h_3=h_5=h_6=0$, we get $r=sC(C^{q+1}+1)^{-1}$, $t=C^q(C^{q+1}+1)^{-1}(s+C^{q+1}+1)$, $u=0$. After substituting them in $h_4$, from $h_4=0$ we get either $s=C^{q+1}+1$ or $s=A^{q+1}(C^{q+1}+1)$. In both cases, a contradiction follows from $h_0=0$.

Now, assume $AE^q + E \neq 0$.
By Remark~\ref{Remark}, if $b_2X_0^2+b_1X_0+b_0$ splits into degree-one factors in $X_0$, then
$$b_2X_0^2+b_1X_0+b_0=(\beta X_0 + \gamma)(\beta (X_0+Z_0) + \gamma),$$
with $\beta^2=b_2$ and $\gamma(\gamma + \beta Z_0) = b_0$. 

Note that $b_0^{(i)}\equiv 0$ if $i\notin \{6,7\}$ and $\beta^{(i)}\equiv 0$ if $i\notin \{2,3\}$. 
This implies $\gamma^{(i)}\equiv 0$ if $i\notin \{3,4\}$. 
From the condition $\gamma(\gamma + \beta Z_0) = b_0$, we obtain the system
\begin{equation}\label{System1}
\begin{cases}
\gamma^{(4)}(\gamma^{(4)}+Z_0 \beta^{(3)})=0\\
Z_0\beta^{(2)}\gamma^{(4)}+Z_0\beta^{(3)}\gamma^{(3)}=b_0^{(7)}\\
\gamma^{(3)}(\gamma^{(3)}+Z_0 \beta^{(2)})=b_0^{(6)}.
\end{cases}
\end{equation}
From the first two equations, we derive
$$
h:=b_0^{(6)}+\frac{b_0^{(7)}}{\beta^{(3)}Z_0}\left(\beta^{(2)}Z_0+\frac{b_0^{(7)}}{\beta^{(3)}Z_0}\right).
$$
After clearing denominators, the numerator of $h$ has coefficient of $Z_0^7Z_1^{5}$ equal to
$$(A^{q+1}+1)^2(C^{q+1}+1)^3(AE^q+E)^{q+1}.$$
By hypothesis, $(A^{q+1}+1)(C^{q+1}+1) \neq 0$, and we have shown that $AE^q+E \neq 0$. 
Therefore, all three factors are non-zero, so this coefficient is non-zero. Thus $h \not\equiv 0$, 
contradicting the requirement that $h \equiv 0$ for a factorization to exist.
This shows that $b_2X_0^2+b_1X_0+b_0$ is absolutely irreducible and has no degree-one factors. 
The variety $\mathcal{Z}$ defined by ($G$ was defined right before Proposition~\ref{Prop:CoeffX_1})
$$\begin{cases}
G(X_0,X_1,Z_0,Z_1)=0,\\
b_2X_0^2+b_1X_0+b_0=0
\end{cases}$$
is a complete intersection in $\mathbb{A}^4$ of two hypersurfaces, hence has dimension $4-2=2$. 
After removing the components $X_0=0$ and $X_0+Z_0=0$, which lie on the forbidden hyperplanes, 
the remaining part of $\mathcal{Z}$ is absolutely irreducible (since $b_2X_0^2+b_1X_0+b_0$ is 
absolutely irreducible). Since both defining equations are fixed by $\phi$, this component is 
$\phi$-fixed. Moreover, it is not contained in any of the forbidden hyperplanes. By 
Theorem~\ref{Th:Key}, for $q \geq 2^{20}$, the function $f_{A,B,C,D,E}(x)$ is not APN.

\item \textbf{Proof of Part (2):} Assume $(A^{q+1}+1)=0$, $(C^{q+1}+1)\neq 0$, $AE^q+E \ne 0$, 
and the cubic $T^3+CT^2+AC^qT+A$ has a root $k \in \mathbb{F}_{q^2}$ with $k \ne 1$ and $k^{q+1}=1$.
In this case, Equation~\eqref{FactorizationS} becomes 
$$
\overline{G}(X_0,Z_0,Z_1)=(A E^q + E)^2X_0^2(X_0 + Z_0)^2(A E^q Z_0^2 + C Z_0 + Z_1)
(A C^q Z_0^2 Z_1 + A Z_0^3 + C Z_0 Z_1^2 + Z_1^3)^2=0.
$$
The cubic factor $P(Z_0,Z_1) := A C^q Z_0^2 Z_1 + A Z_0^3 + C Z_0 Z_1^2 + Z_1^3$ can be 
rewritten (for $Z_0 \neq 0$) by setting $T = Z_1/Z_0$,
$$
P(Z_0,Z_1) = Z_0^3(T^3 + CT^2 + AC^qT + A).
$$
By hypothesis, this cubic in $T$ has a root $k \in \mathbb{F}_{q^2}$ with $k \ne 1$ and $k^{q+1}=1$. 
Since $k^{q+1}=1$, we have $k^q = k^{-1}$, which means the line $\mathcal{L}_k$ defined by 
$Z_1 = kZ_0$ is $\mathbb{F}_q$-rational.

We verify that the plane $\mathcal{P}$ defined by $Z_1=kZ_0$ and $X_1=kX_0$ satisfies the 
first equation of $\mathcal{W}$. Substituting into $F_1$ with $B=0$ and $D=AC^q$,
$$(AZ_0 + k^2Z_0^2E + kZ_0AC^q)X_0^2 + (Z_0^2A + k^2Z_0^2C)X_0 + (Z_0^2E + Z_0C + kZ_0)(kX_0)^2 
+ Z_0^2AC^q(kX_0).$$
Factoring out $Z_0X_0$ and using $k^3 + Ck^2 + AC^qk + A = 0$ (from the cubic), one can verify 
(by algebraic manipulation) that this expression vanishes. A similar verification holds for $F_2$. 
Moreover, $\mathcal{P}$ is fixed by $\phi$ (since the condition $k^{q+1}=1$ ensures invariance).

The plane $\mathcal{P}$ is not contained in any of the forbidden hyperplanes $X_0=0$, $X_1=0$, 
$Z_0=X_0$, $Z_1=X_1$, $Z_0=0$, $Z_1=0$ (for generic $k \neq 0,1$). Therefore, by 
Theorem~\ref{Th:Key}, for $q \geq 2^{20}$, the function $f_{A,B,C,D,E}(x)$ is not APN.

\item \textbf{Proof of Part (3):} Assume $(A^{q+1}+1)\neq 0$ and $(C^{q+1}+1)=0$.
When $(C^{q+1}+1)=0$ (so $C^{q+1}=1$) but $(A^{q+1}+1)\neq 0$, we have $b_0 \equiv 0$ from 
the factor $(C^{q+1}+1)$ in its expression. After clearing denominators, $G(X_0,X_1,Z_0,Z_1)$ 
and $\overline{G}(X_0,Z_0,Z_1)$ both contain the common factor
$$H = A^{q+1}CZ_0 + A^{q+1}Z_1 + AE^qZ_0^2 + A^qCEZ_1^2 + CE^qZ_1^2 + CZ_0 + EZ_0^2 + Z_1.$$
This can be rewritten as
$$H = (CZ_0 + Z_1)(A^{q+1} + 1) + (E + AE^q)Z_0^2 + E^q(A^qC + C)Z_1^2.$$
Since $C^{q+1}=1$, we have $C^q = C^{-1}$. By direct computation, $H$ is invariant under $\phi$.
\begin{itemize}
    \item When $AE^q + A^qC^3E + C^3E^q + E =0$, the polynomial $H$ factors, and the hyperplane 
$CZ_0 + Z_1 = 0$ is a component. This hyperplane $CZ_0 + Z_1 = 0$ is $\phi$-fixed: under 
$\phi$, it becomes $CZ_1 + Z_0 = 0$, which equals $Z_0 + C^qZ_1 = 0$. Since $C^{q+1}=1$, 
we have $C^q = C^{-1}$, so this is $Z_0 + C^{-1}Z_1 = 0$, or equivalently $CZ_0 + Z_1 = 0$.
This hyperplane is not contained in any of the forbidden hyperplanes. By Theorem~\ref{Th:Key}, 
for $q \geq 2^{20}$, the function $f_{A,B,C,D,E}(x)$ is not APN. 
When  $AE^q + A^qC^3E + C^3E^q + E \neq 0$, the polynomial $H$ is absolutely irreducible and thus defines a component of $\mathcal{W}$ invariant under $\phi$. It is clearly not contained in any  of the forbidden hyperplanes. By Theorem~\ref{Th:Key}, 
for $q \geq 2^{20}$, the function $f_{A,B,C,D,E}(x)$ is not APN.
\end{itemize}

\item \textbf{Proof of Part (4):} Assume $(A^{q+1}+1)=0$, $(C^{q+1}+1)=0$, and $AE^q+E \ne 0$.
We have that $\sqrt{A} Z_0 + \sqrt{C}Z_1$ is a common factor of $\overline{G}(X_0,Z_0,Z_1)$ and $G(X_0,X_1,Z_0,Z_1)$. Such a factor defines a hyperplane in $\mathcal{W}$, fixed by $\phi$, and distinct from the forbidden ones. Via Theorem \ref{Th:Key}, $f_{A,B,C,D,E}$ is not APN.
\end{proof}

\begin{prop}\label{B0_2}
Suppose that conditions $(C1)$ and $(C2)$ do not hold and let $B=0$. If $q$ is sufficiently large and $f_{A,B,C,D,E}(x)$ is APN, then $(AC^q+D)E = 0$.
\end{prop}

\begin{proof}
We prove the contrapositive: if $(AC^q+D)E \neq 0$, then $f_{A,B,C,D,E}(x)$ is not APN.

Assume $(AC^q+D)E \ne 0$. By hypothesis, the factorization in \eqref{FactorizationS} reads
$$
X_0(X_0+Z_0)(A^qZ_1 + D^qZ_0 + E^qZ_0^2)(b_2X_0^2+b_1X_0+b_0)=0,
$$
where
\allowdisplaybreaks
\begin{align*}
    b_2:=& \Big ((A^{q+1}+C^{q+1}+D^{q+1}+1)Z_0Z_1 + (AD^q+C)Z_0^2 + (AE^q+E)Z_0^3 \\
    &\quad + (A^qD+C^q)Z_1^2 + (A^qE+E^q)Z_1^3  + (CE^q+D^qE)Z_0Z_1^2 \\
    &\quad + (C^qE+DE^q)Z_0^2Z_1 \Big)^2;\\
    b_1:=& b_2Z_0; \\
    b_0:=& \Big( (A^qC+D^q)Z_1^2 + (A^qD+C^q)Z_0Z_1 + (A^qE+E^q)Z_0Z_1^2 \\
&\quad + (C^{q+1}+ D^{q+1})Z_0^2 + (C^qE+DE^q)Z_0^3  \Big)  
    \Big( (A^{q+1}+C^{q+1}+D^{q+1}+1)Z_0^2Z_1^2 \\
    &\quad + (AC^q+D)Z_0^4 + (A^qC+D^q)Z_1^4 \Big).
\end{align*}
We distinguish two cases based on whether $C^qE+DE^q$ vanishes.

\textbf{Case 1: $C^qE+DE^q = 0$.}
Since $(A^qC+D^q)E \ne 0$ (from our hypothesis $(AC^q+D)E\neq 0$ and taking $q$-th powers), we have $C(AE^q+E) \ne 0$. 
From $C^qE+DE^q = 0$, we get $D = C^qE^{1-q}$ (since $E\neq 0$).
By Remark~\ref{Remark}, if $b_2X_0^2+b_1X_0+b_0$ splits into degree-one factors in $X_0$, then
$$b_2X_0^2+b_1X_0+b_0=(\beta X_0 + \gamma)(\beta(X_0+Z_0) + \gamma),$$
where $\beta^2=b_2$ and $\gamma(\gamma+\beta Z_0) = b_0$.

Note that $b_0^{(i)}\equiv 0$ if $i\notin \{6,7\}$ and $\beta^{(i)}\equiv 0$ if $i\notin \{2,3\}$. This shows that $\gamma^{(i)}\equiv 0$ if $i\notin \{3,4\}$. 
 As established in the proof of Proposition~\ref{B0_1}, the condition $\gamma(\gamma + \beta Z_0) = b_0$ leads to the system
\begin{equation}\tag{\ref{System1}}
\begin{cases}
\gamma^{(4)}(\gamma^{(4)}+Z_0 \beta^{(3)})=0\\
Z_0\beta^{(2)}\gamma^{(4)}+Z_0\beta^{(3)}\gamma^{(3)}=b_0^{(7)}\\
\gamma^{(3)}(\gamma^{(3)}+Z_0 \beta^{(2)})=b_0^{(6)}.
\end{cases}
\end{equation}
From the first two equations of System~\eqref{System1}, we obtain
$$
h:=b_0^{(6)}+\frac{b_0^{(7)}}{\beta^{(3)}Z_0}\left(\beta^{(2)}Z_0+\frac{b_0^{(7)}}{\beta^{(3)}Z_0}\right).
$$
After substituting $D=C^qE^{1-q}$ and clearing denominators, the numerator of $h$ equals
$$Z_1Z_0^3(A^qE + E^q)^{q+1}\Big(EC^q(AE^q + E)Z_0^4 + E^{q+1}(A^{q+1}+1)Z_0^2Z_1^2+CE^q(A^q E + E^q)Z_1^4\Big)^2.$$

Thus, $h$ is not the zero polynomial, which contradicts the requirement that $h\equiv 0$ for a factorization to exist. This shows that $b_2X_0^2+b_1X_0+b_0$ has no degree-one factors.

Consequently, the variety $\mathcal{Z}$ defined by
$$\begin{cases}
G(X_0,X_1,Z_0,Z_1)=0\\
b_2X_0^2+b_1X_0+b_0=0
\end{cases}$$
contains an absolutely irreducible component (after removing the common factors corresponding to $X_0=0$ and $X_0+Z_0=0$). Moreover, both $G$ and the polynomial $b_2X_0^2+b_1X_0+b_0$ are fixed by $\phi$, so $\mathcal{Z}$ is $\phi$-stable and contains a $\phi$-fixed absolutely irreducible component not contained in the forbidden hyperplanes. By Theorem~\ref{Th:Key}, if $q$ is sufficiently large, $f_{A,B,C,D,E}(x)$ is not APN.

\textbf{Case 2: $C^qE+DE^q \ne 0$.}

We construct an explicit $\phi$-fixed component of $\mathcal{W}$ not contained in the forbidden hyperplanes.

Consider the polynomial
\begin{eqnarray*}
    H(X_0,Z_0,Z_1) &:=& \Big(
        (C^2 D E^{2q} + C^{2q} E^2 + D^2 E^{2q} + D^{2q+1} E^2)E^qZ_1\\
        &&+ (C^2 E^{2q} + C^{2q} D^q E^2 + D^{q+2} E^{2q} + D^{2q}E^2)EZ_0\\
        &&+(C^q E + D E^q)^2E^{q+1}Z_0^2  + (C E^q + D^q E)^2E^{q+1}Z_1^2 \Big)X_0\\
        &&+(CE^q + D^q E)\Big((C^2 E^q + C D^q E + C^q E + D E^q)E^qZ_1^2\\
        &&+(CE^q + C^{2q} E + C^q D E^q + D^q E)EZ_0^2\Big).
\end{eqnarray*}

Note that $H(X_0,Z_0,Z_1)$ satisfies:
\begin{enumerate}
\item $H\not\equiv 0$, under our hypotheses $(AC^q+D)E\neq 0$ and $C^qE+DE^q\neq 0$;
\item The variety $\mathcal{C}$ defined by $H(X_0,Z_0,Z_1)=0$ and $\phi(H(X_0,Z_0,Z_1))=0$ is absolutely irreducible and fixed by $\phi$;
\item Direct substitution verifies that $\mathcal{C}\subseteq \mathcal{W}$ (i.e., points on $\mathcal{C}$ satisfy both equations defining $\mathcal{W}$);
\item $\mathcal{C}$ is not contained in any of the forbidden hyperplanes $X_0=0$, $X_1=0$, $Z_0=X_0$, $Z_1=X_1$, $Z_0=0$, $Z_1=0$.
\end{enumerate}
Therefore, by Theorem~\ref{Th:Key}, if $q$ is sufficiently large, $f_{A,B,C,D,E}(x)$ is not APN.

We have thus shown that if $(AC^q+D)E \neq 0$, then in both cases (whether $C^qE+DE^q = 0$ or $C^qE+DE^q \neq 0$), the function $f_{A,B,C,D,E}(x)$ is not APN for $q$ sufficiently large. This completes the proof of  the proposition.
\end{proof}

The case $B=E=0$ was previously investigated \cite{li}, where the authors characterize APN functions with $A, D \in \mathbb{F}_{q^2}$ and $C \in \mathbb{F}_q$, see Theorem 1.2. Subsequently, Chase and Lisoněk in \cite{lisonek} show that if $A, D, C \in \mathbb{F}_{q^2}$, any such APN functions are necessarily equivalent to Gold functions. The characterization result mentioned above is not fully correct; indeed, we found a counterexample. If we consider $q^2=2^8$ and $\alpha$ a root of the polynomial $x^8 + x^4 + x^3 + x^2 + 1$, then $F(x)=x^{3q}+x^{2q+1}+\alpha^8 x^{q+2}+\alpha x^3$ is not APN since the equation $F(x+1) + F(x) = 0$ possesses exactly four solutions. To address this discrepancy, we present the following proposition.

\begin{prop}\label{B0_3}
 Suppose that conditions $(C1)$ and $(C2)$ do not hold. Let $B=E=0$ and $AC^q+D \ne 0$ and $q$ large enough. If $f_{A,B,C,D,E}(x)$ is APN then  one of the following possibly holds:
\begin{enumerate}
    \item[(i)] $A^{q+1}+C^{q+1}+D^{q+1}+1=0$ and $(AC^q+D)^{q-1}=(AD^q+C)^{2q-1}$; or
    \item[(ii)] $A^{q+1}+C^{q+1}+D^{q+1}+1\neq 0$, $C\neq AD^q$, and $p_1p_2=0$, where 
\allowdisplaybreaks
\begin{eqnarray*}
p_1&:=&A^{q+2}C^q + A^2 D^{2q} + A^{q+1} D + A C^{2q+1}\\
&&\quad + A C^q D^{q+1} + A C^q + C^2 + C^{q+1} D + D^{q+2}+ D,\\
p_2&:=&A^{2q+2}C^{q+1} + A^{q+1}(C^{q+1}D^{q+1}+C^{q+1}\\
&&\quad + D^{q+1}+D^{2q+2}+1) + C^{3q+3}+ C^{2q+2}+ C^{q+1}D^{q+1}\\
&&\quad +D^{3q+3} + C^{2q+2}D^{q+1}+ C^{q+1} D^{2q+2} \\
&&\quad +\mathrm{Tr}_{q^2/q}(A^{q+2}(C D^{2q}+ C^q D^q)+ A^2 D^{3q}\\
&&\quad + A (C^{2q+1} D^q+C^{3q}+ C^q D^{2q+1}+C^q D^q)+C^2 D^q  ).
\end{eqnarray*}
\end{enumerate}
\end{prop}

\begin{proof}
By hypothesis, the factorization in \eqref{FactorizationS} reads
      $$
      X_0(X_0+Z_0)(A^qZ_1 + D^qZ_0)(b_2X_0^2+b_1X_0+b_0),
      $$
      where
\begin{align*}
          b_2:=& \Big ((A^{q+1}+C^{q+1}+D^{q+1}+1)Z_0Z_1 + (AD^q+C)Z_0^2 + (A^qD+C^q)Z_1^2 \Big)^2,\\
          b_1:=& b_2Z_0, \\
          b_0:=& \Big( (A^qC+D^q)Z_1^2 + (A^qD+C^q)Z_0Z_1 +  (C^{q+1}+ D^{q+1})Z_0^2 \Big)  \\&\cdot \Big( (A^{q+1}+C^{q+1}+D^{q+1}+1)Z_0^2Z_1^2 + (AC^q+D)Z_0^4 + (A^qC+D^q)Z_1^4   \Big).
\end{align*}
Since $f_{A,B,C,D,E}(x)$ is assumed to be APN, then  $a_2X_0^2+a_1X_0+a_0$ must contain a factor of degree one in $X_0$. 

\item \textbf{Case 1: $A^{q+1}+C^{q+1}+D^{q+1}+1=0$.}
In this case, 
\begin{align*}
          b_2:=& \Big ((AD^q+C)Z_0^2 + (A^qD+C^q)Z_1^2 \Big)^2,\\
          b_1:=& b_2Z_0, \\
          b_0:=& \Big( (A^qC+D^q)Z_1^2 + (A^qD+C^q)Z_0Z_1 +  (C^{q+1}+ D^{q+1})Z_0^2 \Big)  \\
          &\qquad\qquad \cdot \Big(  (AC^q+D)Z_0^4 + (A^qC+D^q)Z_1^4   \Big).
\end{align*} 
By Remark \ref{Remark}, if the polynomial splits in two factors of degree one in $X_0$, then 
$$
b_2X_0^2+b_1X_0+b_0=(\beta X_0 + \gamma)(\beta(X_0+Z_0)+\gamma),
$$
where $\beta^2=b_2$ and $b_0=\gamma(\gamma+ \beta Z_0)$. Since $\beta$ is homogeneous of degree $2$ and $b_0$ is homogeneous of degree $6$, then $\gamma$ must be homogeneous of degree $3$. Put
$$
\gamma =r Z_0^3+sZ_0^2Z_1+tZ_0Z_1^2+uZ_1^3.
$$
The polynomial $h:=b_0+\gamma \beta Z_0 + \gamma^2$ must be the zero polynomial. Let $h:=\sum_{i=0}^6h_iZ_1^iZ_0^{6-i}$, where
\begin{eqnarray*}
    h_0&=& r^2 + r(AD^q + C) + (AC^q+D)(C^{q+1}+D^{q+1}),\\
     h_1&=&  s(AD^q + C) + (AC^q+D)(A^qD+C^q),\\ 
      h_2&=& r(A^qD + C^q) + s^2 + t(AD^q + C) + (AC^q+D)(A^qC+D^q),\\
      h_3&=& s(A^qD + C^q) + u(AD^q + C), \\
     h_4&=& t^2 + t(A^qD + C^q) + (A^qC+D^q)(C^{q+1}+D^{q+1}),\\
      h_5&=& u(A^qD + C^q) + (A^qC+D^q)(A^qD+C^q),\\
       h_6&=& u^2 + A^{2q}C^2 + D^{2q}.
\end{eqnarray*}
Note that if  $C=AD^q$, then $h_2=(AC^q+D)(A^qC+D^q)$ and it cannot vanish by assumption. Thus, we can assume $AD^q+C \ne 0$. From $h_6=0$, we get $u=A^qC+D^q$. From $h_1$ and $h_3$, we get 
$$
s=\frac{(AC^q+D)(A^qD+C^q)}{AD^q+C}=\frac{(AC^q+D)^q(A^qD+C^q)^q}{(AD^q+C)^q},
$$
that is, 
$$(AC^q+D)^{q-1}=(AD^q+C)^{2q-1}.$$

\item \textbf{Case 2: $A^{q+1}+C^{q+1}+D^{q+1}+1\neq 0$.}
Note that $b_2, b_1\not\equiv 0$ in this case. 
If $C=AD^q$, then $A^{q+1}+C^{q+1}+D^{q+1}+1=(A^{q+1}+1)(D^{q+1}+1)\neq 0$. We apply again the same argument as in the previous proofs. If the polynomial $b_2 X_0^2+b_1X_0+b_0$ splits into degree-one factors in $X_0$, then 
$$
b_2 X_0^2+b_1X_0+b_0=(\beta X_0 + \gamma)(\beta (X_0+Z_0) + \gamma),
$$
where $\beta^2=b_2$ and $\gamma$ is homogeneous of degree three. Therefore, $b_0=\gamma \beta Z_0 + \gamma^2$. Consider 
$$\gamma= r Z_0^3+s Z_0^2Z_1+tZ_0Z_1^2+uZ_1^3.$$
Such an $\gamma$ must make $h:=b_0+\gamma \beta Z_0 + \gamma^2$ the zero polynomial in $Z_0$ and $Z_1$.  Let $h:=\sum_{i=0}^6h_iZ_1^iZ_0^{6-i}$, where
\allowdisplaybreaks
\begin{eqnarray*}
    h_0&=& r^2 + r(AD^q + C) + (AC^q+D)(C^{q+1}+D^{q+1}),\\
     h_1&=& r(A^{q+1}+C^{q+1}+D^{q+1}+1)+ s(AD^q + C) + (AC^q+D)(A^qD+C^q),\\ 
      h_2&=& r(A^qD + C^q) + s^2 +s(A^{q+1}+C^{q+1}+D^{q+1}+1) +t(AD^q + C)\\
      &&+ A^{q+1}D^{q+1} + A C^q D^q + A^q C D + C^{2q+2} + C^{q+1} + D^{2q+2},\\
      h_3&=& s(A^qD + C^q) +t(A^{q+1}+C^{q+1}+D^{q+1}+1)+ u(AD^q + C)\\
      &&+(A^{q+1}+C^{q+1}+D^{q+1}+1)(A^qD + C^q), \\
     h_4&=& t^2 + t(A^qD + C^q) + +u(A^{q+1}+C^{q+1}+D^{q+1}+1)+(A^qC+D^q)(A^{q+1} + 1),\\
      h_5&=& (A^qD + C^q)(u + (A^qC+D^q)),\\
       h_6&=& (u + A^{q}C + D^{q})^2.
\end{eqnarray*} 
From  $h_6=0$ we obtain $u = A^qC + D^q$. 
When $C=AD^q$, $h_4=h_3=0$ yields 
\begin{eqnarray*}
    t^2 + D^{2q+1}(A^{q+1}+1)^2=0,\\
    t(A^{q+1}+C^{q+1}+D^{q+1}+1)=0.
\end{eqnarray*}
Thus $t=0$ and $D^{2q+1}(A^{q+1} + 1)^2=0$, a contradiction. Therefore, no factors exist when $C=AD^q$.

Now assume $C\neq AD^q$. 
From $h_1=0$ it follows 
$$s=\frac{r(A^{q+1}+C^{q+1}+D^{q+1}+1)+(AC^q+D)(A^qD+C^q)}{AD^q + C}.$$

From $h_0=0$ we obtain 
$$r^2 = (A D^q + C)r + AC^{2q+1} + A C^q D^{q+1} + C^{q+1} D + D^{q+2}.$$ 

 Combining it with $h_2=0$ we obtain a-degree one polynomial in $r$, whose coefficient is $(A^q D + C^q)^{2q+1}$. Also, $h_3=0$ is another degree-one polynomial in $r$, whose coefficient is $(A^{q+1} + C^{q+1} + D^{q+1} + 1)(A^q D + C^q)\neq 0$. Combining these two equations and  eliminating $r$ yields the condition $p_1p_2=0$.

Conversely, one can verify by direct computation that when $p_1p_2=0$, the system $h_i=0$ for $i=0,1,2,3,4$ admits a solution $(r,s,t)$, confirming that the factorization exists. Therefore, factors of degree one exist if and only if $p_1p_2=0$ when $C\neq AD^q$.
\end{proof}

\section{Case \texorpdfstring{$ B \ne 0$}{B ne 0}}
\label{sec4}

From now on we consider the case $B\neq 0$.
This first result shows that in the general case $f_{A,B,C,D,E}(x)$ is not APN. 
\begin{thm}
\label{Th:SmallestHomogeneousPart}

Suppose that conditions $(C1)$ and $(C2)$ do not hold.   Suppose that 
    \begin{enumerate}
    \item[(C3)] $(A D^q+C,A^{q+1}+C^{q+1}+D^{q+1}+1)\neq (0,0)$; and 
    \item[(C4)] $h_1:= A^{q+1} B^{q+1} + A B^{2q} + A^q B^2 + B^2 C^q D^q + B^{q+1}C^{q+1}+ B^{q+1} D^{q+1} + B^{q+1} + B^{2q} C D\neq 0$.
\end{enumerate}
 Then $a_2X_0^2+a_1X_0+a_0$ has no factor of degree one in $X_0$. Consequently,  $$f_{A,B,C,D,E}(x):= x(Ax^2+Bx^q+Cx^{2q})+x^2(D x^q+E x^{2q})+x^{3q}\in \mathbb{F}_{q^2}[x] $$
is not APN if $q$ is large enough.
\end{thm}
\begin{proof}
It can be easily checked, by Proposition \ref{Prop:CoeffX_1}, that Conditions (C3) and (C4) imply that the coefficient of $X_1$ in $G(X_0, X_1, Z_0, Z_1)$ is non-vanishing.

Consider again Equation \eqref{Equation1}. If it holds for some $f,h\in \overline{\mathbb{F}_q}[Z_0,Z_1]$ then it also holds for the smallest homogeneous parts in Equation \eqref{Equation1}. 
Such a homogeneous part is given by a linear combination (with coefficients $0,1$) of 
\begin{equation}\label{Equation2}
     (g_3^{(L)})^2(f^{(L)})^2, \qquad (g_3^{(L)})^2Z_0f^{(L)}h^{(L)}, \qquad g_1^{(L)}g_2^{(L)}(h^{(L)})^2.
\end{equation}

If there exist polynomials $f$ and $h$ satisfying Equation \eqref{Equation1}, then the smallest homogeneous part $F^{(L)}$ in the left-hand side of Equation \eqref{Equation1} must vanish. 
Let $\alpha^{(L)}=\deg(f^{(L)})$ and $\beta^{(L)}=\deg(h^{(L)})$. Note that Conditions (C3) and (C4) imply also 
 $$(A^q B+B^q,B D^q+B^q C)\neq (0,0).$$
In fact, $(A^q B+B^q,B D^q+B^q C)= (0,0)$ yields either $B=0$ or $A^{q+1}=1$. In the former case $h_1=0$, a contradiction. In the latter case, $(A^q B+B^q,B D^q+B^q C)= (0,0)$ gives $B^{q+1}=0$, so $B=0$, yielding $h_1=0$, again a contradiction.
In particular,
\begin{eqnarray*}
g_1^{(L)}&:=&(A^q B+B^q) Z_1 + (B D^q+B^q C)Z_0;\\
g_2^{(L)}&:= &(AB^q+B)Z_0^3  + (B C^q+B^q D)Z_0^2 Z_1 + (B D^q+ B^q C) Z_0 Z_1^2   + (A^q B+B^q)Z_1^3;\\
g_3^{(L)}&:=& (A D^q+C)Z_0^2 +(A^{q+1}+C^{q+1}+D^{q+1}+1) Z_0 Z_1+ (A^q D+C^q)Z_1^2.
\end{eqnarray*}
We distinguish a few cases.
\begin{enumerate}
    \item $\alpha^{(L)}>\beta^{(L)}$. Then $\deg(g_1^{(L)}g_2^{(L)}(h^{(L)})^2)=4+2\beta^{(L)}$ is lower than $\deg((g_3^{(L)})^2(f^{(L)})^2)=4+2\alpha^{(L)}$ and $\deg((g_3^{(L)})^2Z_0f^{(L)}h^{(L)})=5+\alpha^{(L)}+\beta^{(L)}$, a contradiction to $F^{(L)}\equiv 0$.
    \item $\alpha^{(L)}<\beta^{(L)}$. Then $\deg((g_3^{(L)})^2(f^{(L)})^2)=4+2\alpha^{(L)}$ is lower than $\deg(g_1^{(L)}g_2^{(L)}(h^{(L)})^2)=4+2\beta^{(L)}$ and $\deg((g_3^{(L)})^2Z_0f^{(L)}h^{(L)})=5+\alpha^{(L)}+\beta^{(L)}$, a contradiction to $F^{(L)}\equiv 0$.
    \item $\alpha^{(L)}=\beta^{(L)}$. Then $\deg((g_3^{(L)})^2(f^{(L)})^2)=4+2\alpha^{(L)}=\deg(g_1^{(L)}g_2^{(L)}(h^{(L)})^2)=4+2\beta^{(L)}$ and they are lower than $\deg((g_3^{(L)})^2Z_0f^{(L)}h^{(L)})=5+\alpha^{(L)}+\beta^{(L)}$. In this case $F^{(L)}\equiv 0$ yields
    $$(g_3^{(L)})^2(f^{(L)})^2=g_1^{(L)}g_2^{(L)}(h^{(L)})^2,$$
    and thus $g_1^{(L)}g_2^{(L)}$ must be a square. 
We claim this is impossible. To see this, compute the resultant of $g_1^{(L)}$ and $g_2^{(L)}$ with respect to $Z_0$,
    $$\text{Res}_{Z_0}(g_1^{(L)}, g_2^{(L)}) = (A^q B + B^q)^2 h_1 Z_1^3.$$
If $g_1^{(L)}g_2^{(L)}$ were a square, then $g_1^{(L)}$ and $g_2^{(L)}$ would share a common factor, which would make this resultant vanish. Since $(A^q B + B^q)\neq 0$ (as shown above) and $h_1\neq 0$ by hypothesis (C4), the resultant is non-zero. Therefore $g_1^{(L)}$ and $g_2^{(L)}$ share no common factors, which means their product cannot be a square. This contradiction shows that $F^{(L)}\not \equiv 0$, completing the proof.
\end{enumerate}
The proof is shown.
\end{proof}

In the following series of propositions we consider the remaining cases not covered by Theorem~\ref{Th:SmallestHomogeneousPart}.

\begin{prop}
\label{C7_1}
Suppose that conditions $(C1)$ and $(C2)$ do not hold. Suppose that $h_1=0$ and 
$BC^q + B^q D\neq 0.$

Let $q$ be large enough. If $f_{A,B,C,D,E}(x)$ is APN then one of the following conditions possibly holds
\begin{enumerate}
    \item $C^q=A^qB+A^qD+B^q$, $A^q B + B^q\neq0$, and $B^{q+1}+D^{q+1}+BD^q+B^qD+1=0$;
    \item $E=0$.
\end{enumerate}
\end{prop}
\begin{proof}
Recall that if $f_{A,B,C,D,E}(x)$ is APN then   $a_2X_0^2+a_1X_0+a_0$ possesses a degree-one factor in $X_0$.
By $h_1=0$ one gets 
$$D^q = \frac{A^{q+1}B^{q+1} + A B^{2q} + A^q B^2 + B^{q+1}C^{q+1} + B^{q+1} + B^{2q}CD}
    {B(BC^q + B^qD)}.$$ After substituting it into \eqref{FactorizationS} and raising the denominator, such a factorization reads
\begin{eqnarray*}
    X_0(X_0+Z_0)\Big ((A^{q+1}B^{q+1} + A B^{2q}+A^q B^2 +B^{q+1}C^{q+1}+B^{q+1}+B^{2q}CD)Z_0\\
    + (A^q B^2 C^q  + A^q B^{q+1}D )Z_1 + (B^2 C^q E^q +  B^{q+1}D E^q)Z_0^2  
        \Big)(b_2X_0^2+b_1X_0+b_0),
\end{eqnarray*}
where
{\small
\allowdisplaybreaks
\begin{align*}
    b_2&:=\Big( (AB^2C^qE^q + AB^{q+1}DE^q + B^2C^qE + B^{q+1}DE)Z_0^3\\&+( B^2C^{2q}E + B^2C^qDE^q + B^{q+1}C^qDE + B^{q+1}D^2E^q)Z_0^2Z_1 \\&+ (A^{q+2}B^{q+1} + A^2B^{2q} + A^{q+1}B^2 + AB^{q+1}C^{q+1} + AB^{q+1} + AB^{2q}CD + B^2C^{q+1} + B^{q+1}CD)Z_0^2 \\&+ (A^{q+1}B^{q+1}E + AB^{2q}E + A^qB^2E + B^2C^{q+1}E^q + B^{q+1}C^{q+1}E + B^{q+1}CDE^q + B^{q+1}E + B^{2q}CDE)Z_0Z_1^2\\&+( A^{q+1}B^2C^q + AB^{2q}D + A^qB^2D + B^2C^{1+2q} + B^2C^q + B^{2q}CD^2)Z_0Z_1\\&+(A^qB^2C^qE + A^qB^{q+1}DE + B^2C^qE^q + B^{q+1}DE^q)Z_1^3 \\&+(A^qB^2C^qD + A^qB^{q+1}D^2 + B^2C^{2q} + B^{q+1}C^qD)Z_1^2 \Big);\\
    b_1&:= b_2Z_0;\\
    b_0&:= \Big((A^qB+B^q)Z_1^2 + (BC^q+B^qD)Z_0^2  \Big)\cdot \Big( B(AB^{q} +B)Z_0 + B(BC^q  + B^{q}D)Z_1 +B(AC^q +D)Z_0^2\\ &+(AB^q+B)Z_1^2 + C(BC^{q} +B^qD)Z_1^2\Big)\cdot
\Big((A^{q+1}B^{q+2}+AB^{1+2q}+A^qB^3+ B^{2+q})Z_0 \\ & + (A^qB^3C^q+A^qB^{2+q}D+ B^{2+q}C^q+ B^{1+2q}D)Z_1 + (A^{q+1}B^{q+1}D+AB^{2q}D+A^qB^2D+B^3C^qE^q\\&+B^{2+q}C^qE+B^{2+q}DE^q+ B^2C^{1+2q}+B^{1+2q}DE+B^{q+1}D+B^{2q}CD^2)Z_0^2 + (A^{q+1}B^{q+1}+AB^{2q}\\&+A^qB^2C^{q+1}+A^qB^2+A^qB^{q+1}CD+B^{q+1}C^{q+1}+B^{q+1}+ B^{2q}CD)Z_1^2 + (A^qB^2C^qD+A^qB^{q+1}D^2\\&+B^2C^{2q}+B^{q+1}C^qD)Z_0Z_1 + (A^qB^2C^qE+A^qB^{q+1}DE+B^2C^qE^q+B^{q+1}DE^q)Z_0Z_1^2\\&+(B^2C^{2q}E + B^2C^qDE^q +  
         B^{q+1}C^qDE +  B^{q+1}D^2E^q)Z_0^3   \Big).
\end{align*}
}

By Remark \ref{Remark}, if $b_2X_0^2+b_1X_0+b_0$ has a non trivial factor in $X_0$ then without loss of generality 
$$b_2X_0^2+b_1X_0+b_0=(\beta X_0+\gamma)(\beta(X_0+Z_0)+\gamma),$$
for some $\beta,\gamma \in \overline{\mathbb{F}}_q[Z_0,Z_1]$. We get $\beta^2=b_2$ and $b_0=\gamma(\gamma+ \beta Z_0)$. 
Now, $b_0^{(i)}=0$ if $i\notin \{4,5,6,7\}$ and $b_2^{(i)}=0$ if $i\notin \{4,6\}$. Thus 
$\beta =\beta^{(2)}+\beta^{(3)}, \quad \gamma =\gamma^{(2)}+\gamma^{(3)}+\gamma^{(4)}.$
It follows that
\allowdisplaybreaks
\begin{align*}
0&=\gamma^{(4)}(\gamma^{(4)}+Z_0\beta^{(3)}),\\
b_0^{(7)}&=\gamma^{(3)}Z_0\beta^{(3)}+\gamma^{(4)}Z_0\beta^{(2)},\\
b_0^{(6)}&=\gamma^{(2)}Z_0\beta^{(3)}+(\gamma^{(3)})^2+\gamma^{(3)}Z_0\beta^{(2)},\\
b_0^{(5)}&=Z_0\gamma^{(2)}\beta^{(2)},\\
b_0^{(4)}&=(\gamma^{(2)})^2.
\end{align*}
Combining the first three we get 
$$b_0^{(6)}=\beta^{(2)}Z_0\beta^{(3)}+\frac{(b_0^{(7)})^2}{Z_0^2(\beta^{(3)})^2}+\frac{b_0^{(7)}\beta^{(2)}}{\beta^{(3)}},$$
or equivalently  
$$(b_0^{(6)})^2+b_2^{(4)}Z_0^2b_2^{(6)}+\frac{(b_0^{(7)})^4}{Z_0^4(b_2^{(6)})^2}+\frac{(b_0^{(7)})^2b_2^{(4)}}{b_2^{(6)}}\equiv0.$$
After raising the denominator, the coefficient of $Z_0^2Z_1^{22}$ in the above polynomial equation is 
$$B^8(BC^q + B^q D)^8(A^q E + E^q)^6(A^q B + A^q D + B^q + C^q)^2.$$

By assumption, $BC^q + B^q D\neq 0$. 
\begin{itemize}
\item Suppose that $A^q E = E^q$ and $E\neq 0$. Then $A^{q+1}=1$. Combining this condition with the coefficient of $Z_0^4Z_1^{20}$, we obtain 
$$ E^4(A^q B + B^q)^8  (A B^q + B C^{q+1}+ B + B^q C D)^8.$$
\begin{itemize}
    \item If $A^q B + B^q=0$.  Looking at the coefficient of $Z_0^{24}$, by $A^q B + B^q=0=1+A^{q+1}=A^q E + E^q$, we get $AC^q + D=0$. This is a contradiction to our hypothesis since the coefficient of $X_1$ in $G(X_0,X_1,Z_0,Z_1)$ vanishes. 
    
    \item If  $A B^q + B C^{q+1}+ B + B^q C D=0$ and $A^q B + B^q\neq 0$, then $h_1=0$ yields $B^{q+1}(B C^q + B^qD)(A^q C + D^q)=0$ and thus $A^q C + D^q=0=AC^q+D$. From $A B^q + B C^{q+1}+ B + B^q C D=0$ we get $(C^{q+1} + 1)(A B^q + B)=0$. So, $C^{q+1} + 1=0$ and, from  $A B^q + B C^{q+1}+ B + B^q C D=0$, $A=CD$. This is a contradiction to our hypothesis since the coefficient of $X_1$ in $G(X_0,X_1,Z_0,Z_1)$ vanishes. 
\end{itemize}
\item Suppose now that $C^q=A^q B + A^q D + B^q$. Combining it with $h_1=0$ and with $C=A B^q + A D^q + B$ we obtain that 

$$(AB^q + B)(A^q B + B^q)(B^{q+1} +D^{q+1}+ BD^q + B^q D  + 1)=0$$
If $A^q B + B^q=0$, since $B\neq 0$, 
$$A^{q+1}=1, \quad C^q=A^q D=D/A, \quad C=A D^q$$
and $G(X_0, X_1, Z_0, Z_1)$ vanishes and thus condition $(C1)$ or $(C2)$ holds, a contradiction.
\end{itemize}
The proof is shown.
\end{proof}

\begin{prop}
\label{C7_2}
Suppose that conditions $(C1)$ and $(C2)$ do not hold. Suppose that $h_1=0$ and 
$$BC^q + B^q D= 0.$$
 Then $B=B^qA$, $BE^q + B^qE\neq 0$.  Also  if $f_{A,B,C,D,E}(x)$ is APN then  $B^q T^3+ B^qC T^2 +BC^q T + B $ has no roots in $\mathbb{F}_{q^2}$.
\end{prop}
\begin{proof}
Since $h_1=0$ and $BC^q + B^q D=0$, we obtain $A = B^{1-q}$ and the coefficient of $X_1$ in $G(X_0,X_1,Z_0,Z_1)$ is 
$$(B E^q + B^q E)Z_0^2(B^{q+1}Z_0 + B C^q Z_0^2 + B^qZ_1^2)^2.$$
If $B E^q + B^q E\neq 0$ then the above coefficient is not vanishing. Also,  $\overline{G}(X_0,Z_0,Z_1)$ reads
$$(B E^q + B^q E)^2(B E^q Z_0^2 + B^qCZ_0 + B^q Z_1)(B C^q Z_0^2 Z_1 + B Z_0^3 + B^q C Z_0 Z_1^2 + B^q Z_1^3)^2X_0^2(X_0+Z_0)^2=0.$$

Suppose that there exists $k\in \mathbb{F}_{q^2}$ such that $Z_1+kZ_0$ is fixed by $\phi$ (i.e. $k^{q+1}=1$) and it is factor of  $BC^q Z_0^2 Z_1 + B Z_0^3 + B^q C Z_0 Z_1^2 + B^q Z_1^3$. In other words
$$k^3 B^q + k^2 B^q C + kB C^q + B=0. $$
Then, by direct checking $Z_1+kZ_0=0=X_1+k X_0$ is a plane fixed by $\phi$, contained in $\mathcal{W}$
but not in any forbidden hyperplane. Thus, by Theorem~\ref{Th:Key}, $f_{A,B,C,D,E}(x)$ is not APN. Note that since $BC^q Z_0^2 Z_1 + B Z_0^3 + B^q C Z_0 Z_1^2 + B^q Z_1^3$ is fixed by $\phi$, the existence of at least one factor in $\mathbb{F}_{q^2}$ yields that either it is the unique one with this property and then it is fixed by $\phi$, or all the three factors are defined over $\mathbb{F}_{q^2}$ and at least one of them is fixed by $\phi$. 
\end{proof}


\section{\texorpdfstring{$B \ne 0$}{B ne 0} and (C3) does not hold}


When $(A D^q+C,A^{q+1}+C^{q+1}+D^{q+1}+1)= (0,0)$ and $B\neq 0$,  from $A D^q+C=A^{q+1}+C^{q+1}+D^{q+1}+1=0$, either $A^{q+1}=1$ or $D^{q+1}=1$ holds.

\begin{prop}\label{C6_1}
  Suppose that   $(A D^q+C,A^{q+1}+C^{q+1}+D^{q+1}+1)= (0,0)$, and   $(A B^q + B)(A E^q + E)\neq 0$.
  Let $q$ be large enough. If $f_{A,B,C,D,E}(x)$ is  APN then $D^{q+1}=1$, $A^{q+1} \ne 1$, $BE^q = B^q E$, and $(AE^q+E)^{1-q}=D\sqrt{D}$.
\end{prop}

\begin{proof}
Since $f_{A,B,C,D,E}(x)$ is assumed to be APN and $q$ large enough, $a_2X_0^2+a_1 X_0+a_0 =0$ must have a component of degree $1$ in $X_0$.
If $D^{q+1}\neq 1$ then $A^{q+1}=1$. With these assumptions the factorization in \eqref{FactorizationS} is 
$$X_0(X_0 + Z_0)(A D^q Z_0 + A E^q Z_0^2 + Z_1)(A D^q Z_0 Z_1^2 + A Z_0^3 + D Z_0^2 Z_1 + Z_1^3)(b_2X_0^2+b_1 X_0+b_0),$$
where 
\allowdisplaybreaks
\begin{align*}
    b_2:=& (AE^q + E)^2(A D^q Z_0 Z_1^2 + A Z_0^3 + D Z_0^2 Z_1 + Z_1^3);\\
     b_1:=& b_2Z_0;\\  
     b_0 :=&(A B^q + B)\Big((A B^q +  B)AD^qZ_0 +(A B^q +B)Z_1+ (A B E^q + A B^q E)Z_0^2\\
     &+(A E^q +  E)Z_0 Z_1^2+   (D E + A D E^q)Z_0^3  \Big).
\end{align*}
By Remark \ref{Remark}, if $b_2X_0^2+b_1X_0+b_0$ splits into two factors of degree one in $X_0$, then
$$
b_2X_0^2+b_1X_0+b_0=(\beta X_0 + \gamma)(\beta (X_0+Z_0)+\gamma),
$$
where $\beta, \gamma \in \overline{\mathbb{F}}_q[Z_0,Z_1]$, $b_2=\beta^2$ and $b_0=\gamma (\gamma + \beta Z_0)$. Since $\deg(b_2)=3$ is odd, $b_2$ cannot be a perfect square, yielding a contradiction.

Thus we can assume that $D^{q+1}=1$. With these assumptions the factorization in \eqref{FactorizationS} is 
$$X_0(X_0 + Z_0)(D Z_0^2 + Z_1^2)(A^qD Z_1 + D E^q Z_0^2 + Z_0)(c_2X_0^2+c_1 X_0+c_0),$$ where
\allowdisplaybreaks
\begin{align*}
    c_2:=&(\sqrt{D} Z_0 + Z_1)^2((A E^q+E)Z_0 + (A^q D E+DE^q)Z_1)^2;\\
    c_1:=&c_2Z_0;\\
    c_0:=&\Big(D(A^{q+1}+1) Z_0^2 + (A^{q+1}+1) Z_1^2 + (A B^q+B) Z_0 + D(A^q B +B^q) Z_1   \Big)\cdot\\
    &\Big(D(A^{q+1}+1+B E^q+B^qE) Z_0^2 + (A^{q+1}+1) Z_1^2 + (A B^q+B) Z_0\\
    &+ D(A^q B  +B^q)Z_1 + (A^q D^2 E +D^2 E^q)Z_0^3 + D(A^q  E +E^q)Z_0 Z_1^2 \Big). 
\end{align*}

If $c_2X_0^2+c_1X_0+c_0$ splits into two factors of degree one in $X_0$, then
$$
c_2X_0^2+c_1X_0+c_0=(\beta X_0+\gamma)(\beta (X_0+Z_0)+\gamma),
$$
where $\beta, \gamma \in \overline{\mathbb{F}_q}[Z_0,Z_1]$, $c_2=\beta^2$ and $c_0=\gamma(\gamma + \beta Z_0)$. Note that $c_0=c_0^{(2)}+c_0^{(3)}+c_0^{(4)}+c_0^{(5)}$, and thus $\gamma=\gamma^{(1)}+\gamma^{(2)}$. Since 
$$
c_0^{(2)}=\Big((A B^q+B) Z_0 + D(A^q B  + B^q)Z_1\Big)^2,
$$
then $\gamma^{(1)}=(A B^q+B) Z_0 + D(A^q B  + B^q)Z_1$.

The two factors of $c_2X_0^2+c_1X_0+c_0$ are
\begin{align*}
   F_1=&(\sqrt{D} Z_0 + Z_1)\Big((A E^q+E)Z_0 + D(A^q  E+E^q)Z_1\Big)(X_0+Z_0)\\
   &+\gamma^{(2)}+ (A B^q+B) Z_0 + D(A^q B  + B^q)Z_1;\\
   F_2=&(\sqrt{D} Z_0 + Z_1)\Big((A E^q+E)Z_0 + D(A^q  E+E^q)Z_1\Big)X_0\\
   &+\gamma^{(2)}+ (A B^q+B) Z_0 + D(A^q B  + B^q)Z_1.
\end{align*}

Now, the homogeneous part of degree $3$ in $F_1F_2$ vanishes and thus, comparing it with the one in $c_2X_0^2+c_1X_0+c_0$, 
$$D (B E^q + B^qE)Z_0^2\Big((A B^q+B) Z_0 + D(A^q B +B^q) Z_1\Big)\equiv0,$$
and thus  $B E^q + B^qE=0$ (recall that $E\neq 0$ since $A E^q + E\neq 0$). 
Suppose that $A^{q+1}\neq 1$. Then
\begin{eqnarray*}
   (A B^q+B) Z_0 + D(A^q B +B^q)Z_1 &=&(A BE^{q-1}+B) Z_0 + D(A^q B +B E^{q-1})Z_1\\
   &=&B\Big((A E^{q-1}+1) Z_0 + D(A^q + E^{q-1})Z_1\Big). 
\end{eqnarray*}

Since the homogeneous part of degree $5$ in  $c_2X_0^2+c_1X_0+c_0$ is 
$$(\sqrt{D} Z_0 + Z_1)\Big((A E^q+E)Z_0 + D(A^q  E+E^q)Z_1\Big)\gamma^{(2)}Z_0=D(A^q E + E^q)(A^{q+1} + 1)(\sqrt{D} Z_0 + Z_1)^4Z_0,$$

$$\gamma^{(2)}=\lambda (\sqrt{D} Z_0 + Z_1)^2, \qquad (A E^q+E)Z_0 + D(A^q  E+E^q)Z_1=\mu (\sqrt{D} Z_0 + Z_1),$$
where $\lambda,\mu \in \overline{\mathbb{F}}_q^*$. Therefore, we have the condition $AE^q+E=D\sqrt{D}(A^qE+E^q)$.

Consider now $D^{q+1}=A^{q+1}=1$. With these assumptions the factorization in \eqref{FactorizationS} is 
$$
X_0(X_0+Z_0)(\sqrt{D}Z_0 + Z_1)^2(AZ_0 + DZ_1)(ADE^qZ_0^2 + AZ_0 + DZ_1)(d_2X_0^2+d_1X_0+d_0)=0,
$$
where
\allowdisplaybreaks
\begin{align*}
    d_2:=& (AE^q + E)(\sqrt{D}Z_0 + Z_1)^2(AZ_0 + DZ_1);\\
    d_1:=&d_2Z_0;\\
    d_0:=& (AB^q+B) \Big( A(AB^q+ B)Z_0 + D(AB^q +B)Z_1 + AD(BE^q  + B^qE)Z_0^2\\ &+ D(AE^q+E)Z_0Z_1^2 + D^2(AE^q
       + E)Z_0^3  \Big).
\end{align*}
Arguing as in the previous cases, we get a contradiction since $\deg(d_2)=3$ and $\gcd(d_2,d_0)=1$.
\end{proof}

\begin{prop}\label{C6_2}
  Suppose that    $(A D^q+C,A^{q+1}+C^{q+1}+D^{q+1}+1)= (0,0)$, and $A B^q + B\neq 0$, $A E^q + E= 0$.
  Let $q$ be large enough.

 Then $f_{A,B,C,D,E}$ is APN if, and only if, one of the following holds:
    \begin{enumerate}
        \item $A B^q + B\neq 0$, $A E^q + E= 0$, $C=AD^q, A^{q+1}=1, D \ne 0, D^{q+1} \ne 1$, and $T^3+AD^qT^2+DT+A$ has no roots in $\F_{q^2}$; 
        \item $A B^q + B\neq 0$, $A E^q + E= 0$,  $D=C=0$, $A^{q+1}=1$, $A \ne 1$, and $q \equiv 2 \pmod{3}.$
    \end{enumerate}
\end{prop}
\begin{proof}
By $AE^q+E=0$ we have either $E=0$ or $A^{q+1}=1$.
If $A^{q+1}=1$, then, after clearing the denominators
\begin{align*}
    G(X_0,X_1,Z_0,Z_1)&=
    (A B^q + B)(X_0 Z_1 + X_1 Z_0)(A D^q Z_0 + Z_1)\\
    \overline{G}(X_0,Z_0,Z_1)&=(A B^q + B)X_0(X_0+Z_0)(A D^q Z_0 + Z_1)^2(A D^q Z_0 Z_1^2 + A Z_0^3 + D Z_0^2 Z_1 + Z_1^3)
\end{align*}
and $A D^q Z_0 + Z_1$ is a common absolutely irreducible factor.

\begin{itemize}
    \item If $D^{q+1}=1$ then $A D^q Z_0 + Z_1$ is fixed by $\phi$ and $\mathcal{W}$ contains a hyperplane fixed by $\phi$ and different from the forbidden ones. By Theorem \ref{Th:Key}, if $q$ is large enough, $f_{A,B,C,D,E}(x)$ is not APN.
    \item If $D^{q+1}\neq 1$ and  $D \neq 0$ then $A D^q Z_0 + Z_1$ is not fixed by $\phi$. By direct computation $A D^q Z_0 Z_1^2 + A Z_0^3 + D Z_0^2 Z_1 + Z_1^3$ is fixed by $\phi$ and homogeneous of degree $3$ and  it splits into three different factors of degree 1. Let 
    $$A D^q Z_0 Z_1^2 + A Z_0^3 + D Z_0^2 Z_1 + Z_1^3=(Z_1+k_1Z_0)(Z_1+k_2Z_0)(Z_1+k_3Z_0),$$
for some $k_i\in \overline{\mathbb{F}_{q}}^*$ and $k_i $ are solutions of $ T^3+A D^q T^2+D T+A=0$.

\begin{enumerate}
    \item  If $k_1\in \mathbb{F}_{q^2}$, then either $\phi(Z_1+k_1Z_0)=Z_0+k_1^qZ_1$ is divisible by $Z_1+k_1Z_0$ and thus $Z_1+k_1Z_0$ is fixed by $\phi$ or $(Z_1+k_1Z_0)\mid \phi(\phi(Z_1+k_1Z_0))$. In the former case $Z_1+k_1Z_0$ defined a factor $\overline{G}$ fixed by $\phi$. In the latter case the third factor is defined over $\mathbb{F}_{q^2}$ and fixed by $\phi$. In both cases, there exists a variety in $\mathcal{W}$ defined by 
$$Z_1+kZ_0=0=X_1+kX_0,$$
for some $k\in \mathbb{F}_{q^2}$ fixed by $\phi$, absolutely irreducible, and not contained in any forbidden hyperplane. By Theorem \ref{Th:Key} $f_{A,B,C,D,E}(x)$ is not APN, if $q$ is large enough.
\item If $k_i\notin \mathbb{F}_{q^2}$ for each $i=1,2,3$, then the  solutions of $\overline{G}(X_0,Z_0,Z_1)=0$ satisfy $X_0=0$ or $X_0+Z_0=0$ or $Z_0=0=Z_1$ and they are all contained in the forbidden hyperplanes. By Theorem \ref{Th:Key} $f_{A,B,C,D,E}(x)$ is  APN.
\end{enumerate}
    
    \item If $D= 0$ then $C=0$. In this case 
    \begin{eqnarray*}
     G(X_0,X_1,Z_0,Z_1)&=&
    (A B^q + B)(X_0 Z_1 + X_1 Z_0)Z_1\\
    \overline{G}(X_0,Z_0,Z_1)&=&(A B^q + B)X_0(X_0+Z_0)Z_1^2(A Z_0^3 + Z_1^3),
    \end{eqnarray*}
    and we distinguish two cases:
    \begin{itemize}
         \item $q\equiv 2 \pmod 3$ and $A\neq 1$. Then $A$ is not a cube in $\mathbb{F}_{q^2}$ and the polynomial $A Z_0^3 + Z_1^3$ is irreducible over $\mathbb{F}_{q^2}$. Its unique solution over $\mathbb{F}_{q^2}$ is $(0,0)$ and thus the solutions of $\overline{G}=0$ are contained in the forbidden hyperplanes. By Theorem \ref{Th:Key}, $f_{A,B,C,D,E}(x)$ is APN.

        \item $q\equiv 1 \pmod 3$ or $A=1$. Then $A$ is a cube in $\mathbb{F}_{q^2}$. When $A$ is a cube in $\mathbb{F}_{q^2}$, the polynomial $A Z_0^3 + Z_1^3$ factors as $(Z_1 + \sqrt[3]{A}Z_0)(Z_1^2 + \sqrt[3]{A}Z_0Z_1 + \sqrt[3]{A^2}Z_0^2)$ (or into three linear factors depending on the field), and the surface $Z_1=\sqrt[3]{A}Z_0$, $X_1=\sqrt[3]{A}X_0$ is not fully contained in the forbidden hyperplanes, so by Theorem \ref{Th:Key}, $f_{A,B,C,D,E}(x)$ is not APN.
    \end{itemize}

\end{itemize}

Suppose now that $E=0$ (and $A^{q+1}\neq 1$). Then, after clearing the denominators
\begin{eqnarray*}
    G(X_0,X_1,Z_0,Z_1)&=&
    (X_0 Z_1 + X_1 Z_0)\\
    &&(A^{q+1}D Z_0^2 + A^{q+1} Z_1^2 + A B^q Z_0 + A^q B D Z_1 + B Z_0 + B^q D Z_1 + 
        D Z_0^2 + Z_1^2)\\
    \overline{G}(X_0,Z_0,Z_1)&=&X_0(X_0+Z_0)(D Z_0^2 + Z_1^2)(A^q D Z_1 + Z_0)\\
    &&(A^{q+1}D Z_0^2 + A^{q+1} Z_1^2 + A B^q Z_0 + A^q B D Z_1 + B Z_0 + B^q D Z_1 + 
        D Z_0^2 + Z_1^2)^2
\end{eqnarray*}
and 
$$H := A^{q+1}D Z_0^2 + A^{q+1} Z_1^2 + A B^q Z_0 + A^q B D Z_1 + B Z_0 + B^q D Z_1 + 
D Z_0^2 + Z_1^2$$ is a common factor fixed, by our assumptions, by $\phi$. 
Since $H= H^{(1)}+H^{(2)}$, $H$ is absolutely irreducible if and only if $\gcd(H^{(1)},H^{(2)})=1$. This happens if and only if   
$(A^2 B^{2q} + A^{2q} B^2 D^3 + B^2 + B^{2q} D^3)(A^{q+1}+1).$
On the other hand if this happens, i.e. $A^2 B^{2q} + A^{2q} B^2 D^3 + B^2 + B^{2q} D^3=0$,  $H^{(1)}\mid H$, and $H^{(1)}$ is not vanishing. Also $H^{(1)}=(A B^q+B)Z_0 + D(A^q B  + B^q)Z_1$ is itself fixed by $\phi$ and thus it defines a hyperplane fixed by $\phi$. In this case both the coefficient of $Z_1$ and $Z_0$ are nonvanishing and   
this means that $\mathcal{W}$ contains a hyperplane fixed by $\phi$ and different from the forbidden ones. By Theorem \ref{Th:Key}, if $q$ is large enough, $f_{A,B,C,D,E}(x)$ is not APN.
\end{proof}

\begin{cor}  
\label{cor:C7_application}
Suppose that $h_1=0$ and $BC^q + B^q D\neq 0$.
If $\gcd(a_2,a_0) \neq 1$, let $\ell = \gcd(a_2,a_0)$ and consider the variety $\mathcal{C}_0$ defined by 
$$\begin{cases} 
G(X_0,X_1,Z_0,Z_1)=0\\ 
\ell(Z_0,Z_1)=0. 
\end{cases}$$
Both $G$ and $\ell$ are fixed by $\phi$, hence $\mathcal{C}_0$ is fixed by $\phi$.
If $\mathcal{C}_0$ contains a $\phi$-fixed absolutely irreducible component $\mathcal{C}$ of dimension $2$ with $\mathcal{C} \not\subseteq \pi_1 \cup \pi_2$, then for $q \geq 2^{20}$, the function $f_{A,B,C,D,E}$ is not APN.
\end{cor}

\begin{proof}
From Section~\ref{sec2}, the variety $\mathcal{W}$ is defined by the system
$$\begin{cases}
(AZ_0+Z_1^2 E+Z_1 D)X_0^2 + (Z_0^2A + Z_1^2C + Z_1B)X_0\\
\hspace{1 cm}+ (Z_0^2E + Z_0C + Z_1)X_1^2 + (Z_0^2D + Z_0B + Z_1^2)X_1=0\\
(A^qZ_1+Z_0^{2}E^q+Z_0D^q)X_1^2 + (Z_1^2A^q + Z_0^{2}C^q + Z_0B^q)X_1\\ 
\hspace{1 cm}+ (Z_1^2E^q + Z_1C^q + Z_0)X_0^2 + (Z_1^2D^q + Z_1 B^q + Z_0^{2})X_0=0.
\end{cases}$$

Under the conditions $h_1=0$ and $BC^q + B^q D\neq 0$, eliminating $X_1$ using $G(X_0,X_1,Z_0,Z_1)$ yields (from Equation~\eqref{FactorizationS}),
$$(Z_0^2E^q + Z_0D^q + Z_1A^q)X_0(X_0 + Z_0)(a_2 X_0^2+a_1X_0+a_0)=0,$$
where $a_2 = g_3^2$ and $a_0 = g_1 \cdot g_2$.
The polynomial $G(X_0,X_1,Z_0,Z_1)$ is constructed to be fixed by $\phi$ (as verified in Section~\ref{sec2}). Since $\ell = \gcd(a_2,a_0)$ is a common factor of $a_2$ and $a_0$ (which are themselves fixed by $\phi$ when constructed from the defining equations), $\ell$ is also fixed by $\phi$.
Therefore, the variety $\mathcal{C}_0$ defined by $G=0$ and $\ell=0$ is fixed by $\phi$, meaning $\phi(\mathcal{C}_0) = \mathcal{C}_0$ as a set.

By hypothesis, $\mathcal{C}_0$ contains a $\phi$-fixed absolutely irreducible component $\mathcal{C}$ of dimension 2 with $\mathcal{C} \not\subseteq \pi_1 \cup \pi_2$, where $\pi_1 : X_0=X_1=0$ and $\pi_2: X_0+Z_0=X_1+Z_1=0$ are the forbidden planes corresponding to trivial APN solutions.
Since $\mathcal{C} \subseteq \mathcal{C}_0 \subseteq \mathcal{W}$ and $\mathcal{C}$ satisfies all the conditions of Theorem~\ref{Th:Key}, we conclude that $f_{A,B,C,D,E}$ is not APN for $q \geq 2^{20}$.
\end{proof}

\begin{rem}
\label{rem6.4}
The corollary provides a practical criterion: when $h_1=0$ and $BC^q + B^q D\neq 0$, we compute $\gcd(a_2,a_0)$. If this gcd is non-trivial, the variety $\mathcal{C}_0$ often contains components satisfying the geometric conditions, leading to non-APN functions. However, exceptional cases exist where all $\phi$-fixed components lie on the forbidden planes $\pi_1 \cup \pi_2$, and these can be APN.

Computational verification for $q \in \{2, 4\}$ shows:
\begin{itemize}
\item When the hypothesis of the corollary holds (i.e., $\gcd(a_2,a_0) \neq 1$ with $\mathcal{C} \not\subseteq \pi_1 \cup \pi_2$), all tested functions are non-APN;

\item When $\gcd(a_2,a_0) \neq 1$ but $\mathcal{C}_0 \subseteq \pi_1 \cup \pi_2$, some functions are APN (exceptional cases). For $q=2$, we found exactly $16$ such APN functions, all satisfying $C=0$ with specific relationships between $A, B, D, E$ that force all $\phi$-fixed components onto the forbidden planes. These are displayed in Table~\textup{\ref{tab:apn_q2}}.

\item When $\gcd(a_2,a_0) = 1$, all tested functions are non-APN, though this case is not covered by Corollary~\textup{\ref{cor:C7_application}}.
\end{itemize}

For $q=2$, among $288$ tuples satisfying $h_1=0$ and $BC^q+B^qD\neq 0$, we found $244\ (84.7\%)$ satisfy the corollary's hypothesis and are non-APN, $16\ (5.6\%)$ are exceptional APN cases with $\mathcal{C}_0 \subseteq \pi_1 \cup \pi_2$, and $28\ (9.7\%)$ have $\gcd(a_2,a_0) = 1$ and are non-APN.

For $q=4$, similar patterns hold with $9120$ APN functions found, all in the exceptional category. A snapshot is shown in Table~{\ref{tab:apn_q4}}. The computational verification code is available at~\textup{\cite{GithubPS25}}.
\end{rem}

\begin{prop}\label{C6_3}
  Suppose that:
  \begin{enumerate}
      \item $(A D^q+C,A^{q+1}+C^{q+1}+D^{q+1}+1)= (0,0)$, and 
      \item $B(A E^q + E) \ne 0$, $A B^q + B= 0$.
  \end{enumerate}
   Let $q$ be large enough. If $f_{A,B,C,D,E}(x)$ is APN then 
   \begin{center}
   $A^{q+1}=1$, $A D^q=C$, $B(A E^q + E) \ne 0$, $A B^q + B= 0$,\\ and $T^3 +A D^q T^2  + D T + A$ has no solutions in $\mathbb{F}_{q^2}$.
   \end{center}
  \end{prop}  

  \begin{proof}
Since $AB^q+B=0$ and $B\ne 0$, we obtain $A^{q+1}=1$.
Then, after clearing the denominators
\begin{eqnarray*}
    G(X_0,X_1,Z_0,Z_1)&=&
    (AE^q + E)\cdot\\
    &&\cdot(A B^q X_0 Z_0^2 Z_1 + A B^q X_1 Z_0^3 + A D^q X_0^2 Z_0 Z_1^2 + A D^q X_0 Z_0^2 Z_1^2\\
    &&+ 
        A X_0^2 Z_0^3 + A X_0 Z_0^4 + D X_0^2 Z_0^2 Z_1 + D X_1 Z_0^4 + X_0^2 Z_1^3 + 
        X_1 Z_0^2 Z_1^2)\\
    \overline{G}(X_0,Z_0,Z_1)&=&(AE^q + E)^2X_0^2(X_0+Z_0)^2(A D^q Z_0 + A E^q Z_0^2 + Z_1)\\
    &&(Z_1^3 +A D^q Z_0 Z_1^2  + D Z_0^2 Z_1 + A Z_0^3 )^2
\end{eqnarray*}
Solutions of $\overline{G}(X_0,Z_0,Z_1)=0$ not contained in the forbidden hyperplanes can arise only from the factors $A D^q Z_0 + A E^q Z_0^2 + Z_1$ and $Z_1^3 +A D^q Z_0 Z_1^2  + D Z_0^2 Z_1 + A Z_0^3$.

Let us consider the factor $$Z_1^3 +A D^q Z_0 Z_1^2  + D Z_0^2 Z_1 + A Z_0^3.$$
We can argue as in the proof of Case (i) of Proposition \ref{C6_2}. Suppose that $Z_1+kZ_0$ is a factor of it, fixed by $\phi$ (in particular $k^{q+1}=1$). The existence of such a factor is equivalent to require that $T^3 +A D^q T^2  + D T + A$ has a root in $\mathbb{F}_{q^2}$. Then the plane $Z_1+kZ_0=0=X_1+kX_0$, by direct computation, is a component of $\mathcal{W}$ and it is fixed by $\phi$. Also, it is not contained in one of the forbidden hyperplanes. Thus, if $q$ is large enough, $f_{A,B,C,D,E}(x)$ is not APN by Theorem \ref{Th:Key}.
\end{proof}

The main results of our investigation are summarized in the following theorem.

\begin{thm}
\label{thm:summary6.11}
    If $f_{A,B,C,D,E}$ is APN then one of the following (possibly) occurs
    \begin{enumerate}
    \item (Proposition \textup{\ref{Prop:Condition_(C1)}} Condition $(C1)$ and $A^{q+1}\neq 1$;  \textbf{[Necessary and sufficient]}
 \item (Proposition \textup{\ref{B0_1}})  $B=AC^q+D=0$, $AE^q+E \ne 0$, $(A^{q+1}+1)(C^{q+1}+1) = 0$;
    \item (Proposition \textup{\ref{B0_3}})  $B=E=0$, $AC^q+D \ne 0$, 
$A^{q+1}+C^{q+1}+D^{q+1}+1=0$ and $(AC^q+D)^{q-1}=(AD^q+C)^{2q-1}$;

 \item (Proposition \textup{\ref{B0_3}})  $B=E=0$, $AC^q+D \ne 0$, $A^{q+1}+C^{q+1}+D^{q+1}+1\neq 0$, $C\neq AD^q$, and $p_1p_2=0$;

    \item (Proposition \textup{\ref{C7_1}})  $h_1=0$, $BC^q + B^q D\neq 0$,  $C^q=A^qB+A^qD+B^q$ and $B^{q+1}+D^{q+1}+BD^q+B^qD+1=0$;
    \item (Proposition \textup{\ref{C7_1}})  $h_1=0$, $BC^q + B^q D\neq 0$,  $E=0$;
    \item (Proposition \textup{\ref{C7_2}})
$h_1=0$, $BC^q + B^q D= 0$, $B=B^qA$, $BE^q + B^qE\neq 0$ and $B^q Z_1^3+ B^q T^2 +C^q T + B $ has no root in $\mathbb{F}_{q^2}$;
    \item (Proposition~\textup{\ref{C6_1}})  $AD^q=C$, $(A B^q + B)(A E^q + E)\neq 0$, $D^{q+1}=1$, $A^{q+1} \ne 1$, $BE^q = B^q E$, and $(AE^q+E)^{1-q}=D\sqrt{D}$;
    \item (Proposition \textup{\ref{C6_2}}) $q\equiv 1 \pmod 3,$ $C=D=0=A^{q+1}+1=AE^q+E$,  $AB^q\neq B$;  \textbf{[Necessary and sufficient]} 
      \item (Proposition \textup{\ref{C6_2}})  $C=AD^q$,  $A^{q+1}+1=AE^q+E=0$,  $AB^q\neq B$,  $D(D^{q+1}+1)\neq 0$,  $T^3+A D^q T^2+D T+A$  has no roots in  $\mathbb{F}_{q^2}$; \textbf{[Necessary and sufficient]} 
     \item (Proposition \textup{\ref{C6_3}})  $A^{q+1}=1$, $A D^q=C$, $B(A E^q + E) \ne 0$, $A B^q + B= 0$, and $T^3 +A D^q T^2  + D T + A$ has no solutions in $\mathbb{F}_{q^2}$.
   
    \end{enumerate}
\end{thm}

\section{Computational verification and discovery of new APN classes}
\label{sec:computational_short}

To complement our theoretical analysis, we conducted extensive computational searches for APN functions within Dillon's family. These computations serve dual purposes: (1) verifying that our theoretical obstructions correctly predict non-APN behavior in the vast majority of cases, and (2) discovering which rare parameter configurations actually yield APN hexanomials, thereby revealing the true diversity within this family.
 For small fields ($q \in \{2,4\}$), exhaustive enumeration over all $(q^2)^5$ tuples $(A,B,C,D,E) \in (\mathbb{F}_{q^2})^5$ is computationally feasible. For each candidate satisfying $A \neq 0$ and avoiding conditions (C1) and (C2) from Proposition~\ref{Prop:CoeffX_1}, we tested the APN property by verifying that
$$f(x+a)+f(x)=f(y+a)+f(y)$$
admits only trivial solutions ($x=y$ or $x=y+a$) for all $a \in \mathbb{F}_{q^2}^*$ and $x, y \in \mathbb{F}_{q^2}$.

For larger fields ($q \in \{8,16\}$), exhaustive search becomes computationally prohibitive, so we employed random sampling of the parameter space. We prioritized parameters avoiding the generic obstruction of Theorem~\ref{Th:SmallestHomogeneousPart} and the conditions of Propositions~\ref{B0_1}--\ref{B0_3}, focusing our search on the exceptional regimes identified by our theoretical analysis.

To assess the diversity of discovered functions, we performed complete CCZ-equivalence classification using Magma implementations. For each field size, we tested pairwise CCZ-equivalence among all discovered APN functions to obtain exact counts of inequivalent classes.

\begin{thm}[Computational Classification]
\label{thm:ccz_classification}
Computational searches yield the following APN hexanomials:
\begin{enumerate}
\item  For $q=2$ (exhaustive): $390$ APN functions in exactly $1$ CCZ-equivalence class, all CCZ-equivalent to the Budaghyan-Carlet (BC) family~\textup{\cite{BC08}}.

\item  For $q=4$ (exhaustive): $28,170$ APN functions in exactly $1$ CCZ-equivalence class, all CCZ-equivalent to the BC family.

\item  For $q=8$ ($120,000$ random candidates): $218$ APN functions in exactly $3$ distinct CCZ-equivalence classes, all outside of the BC family.

\item  For $q=16$ ($200,000$ random candidates): $34$ APN functions in exactly  $2$ distinct CCZ-equivalence classes (verified by complete pairwise testing), both outside of the BC family.
\end{enumerate}
Representatives for each class appear in Tables~\textup{\ref{tab:ccz_F4}--\ref{tab:ccz_F256}}. Remarkably, none of the discovered APN functions are permutations.
\end{thm}

\begin{proof}
The enumeration and classification were performed using SageMath implementations of the algorithms described above, followed by complete CCZ-equivalence testing using Magma. Complete computational code and output files are available at~\cite{GithubPS25}.
\end{proof}

\textbf{Interpretation of results.} The computational results strongly support our theoretical predictions while revealing unexpected richness. The dramatic decrease in APN instances as field size grows—from $28,170$ at $q=4$ to only $9$ at $q=16$—confirms that our obstructions successfully exclude the vast majority of coefficient choices. The parameter space itself grows exponentially (from $2^{20} \approx 10^6$ configurations at $q=2$ to $2^{80} \approx 10^{24}$ at $q=16$), yet APN functions become exponentially rarer, indicating they satisfy very special algebraic constraints.

Yet within this rarefied landscape, we find an interesting pattern. For the exhaustive searches at $q \in \{2,4\}$, all discovered APN functions belong to a single CCZ-equivalence class—the known Budaghyan-Carlet family. This suggests that for small fields, the BC construction captures the entire APN behavior within Dillon's hexanomial framework. However, for larger fields $q \in \{8,16\}$, we discover new CCZ-equivalence classes outside the BC family, with $3$ classes for $q=8$ and $2$ classes for $q=16$. This validates that Dillon's hexanomial family does contain inequivalent APN functions beyond the BC construction, though they appear only in larger fields.

The universal absence of permutations among all tested APN hexanomials is particularly striking. This distinguishes Dillon's family from other APN constructions where permutation polynomials exist (albeit on odd dimensions), suggesting these hexanomials possess structural features fundamentally incompatible with bijectivity. Understanding this phenomenon could provide insight into the relationship between the APN property and injectivity in polynomial mappings over finite fields.

The computational searches also validate the precision of our theoretical obstructions. Parameters satisfying the hypotheses of Theorem~\ref{Th:SmallestHomogeneousPart}, Corollary~\ref{cor:C7_application}, or Propositions~\ref{B0_1}--\ref{B0_3} consistently yield non-APN functions, demonstrating negligible false positive rates. Conversely, the rare APN instances concentrate precisely in the exceptional regimes our theory identified: cases where $h_1=0$ with special GCD structure, degenerate situations where condition (C6) fails, and boundary configurations involving irreducible cubic polynomials. This tight correspondence between theoretical predictions and computational observations suggests our case analysis has captured the essential structure of the APN landscape.

\section{Conclusions and Future Directions}
\label{sec7}

We have undertaken a systematic investigation of Dillon's hexanomial functions over $\mathbb{F}_{q^2}$, where $q=2^n$, of the form
$$f_{A,B,C,D,E}(x) = x(Ax^2 + Bx^q + Cx^{2q}) + x^2(Dx^q + Ex^{2q}) + x^{3q}.$$
By reformulating the APN condition as a problem concerning algebraic varieties over finite fields, we have established comprehensive necessary conditions for these functions to achieve almost perfect nonlinearity. Our exhaustive case-by-case analysis reveals that the vast majority of Dillon's hexanomials fail to be APN due to specific algebraic and geometric obstructions.

\subsection{Main results}
The heart of our approach lies in Theorem~\ref{Th:Key}, which transforms the combinatorial problem of counting solutions to differential equations into a geometric question about algebraic varieties. For $q \geq 2^{20}$, we prove that if the associated variety $\mathcal{W}$ contains an absolutely irreducible $\phi$-fixed component not contained in certain forbidden hyperplanes, then the function cannot be APN. This geometric reformulation allows us to harness powerful tools from algebraic geometry -- the Cafure-Matera bounds, Lang-Weil estimates, and resultant theory -- to systematically exclude large regions of the parameter space.

Our investigation naturally divided into cases based on whether $B=0$ or $B \neq 0$. When $B=0$, Propositions~\ref{B0_1}--\ref{B0_3} show that APN behavior is possible only under very restrictive conditions, often involving cubic polynomials having no roots in $\mathbb{F}_{q^2}$. When $(AC^q+D)E \neq 0$, the function is always non-APN through explicit gcd arguments or by constructing $\phi$-fixed components that violate the geometric criterion.

The case $B \neq 0$ proved more intricate. Theorem~\ref{Th:SmallestHomogeneousPart} applies when both condition (C6) and the non-vanishing of $h_1$ hold, employing an argument that examines the lowest homogeneous parts of polynomials $g_1$ and $g_2$ in the factorization of $a_0$. By showing their resultant is non-zero, we prove their product cannot be a perfect square, preventing the existence of degree-one factors in $X_0$. This obstruction excludes a generic, high-dimensional subset of the parameter space from containing APN functions.

When $h_1=0$, the situation becomes more delicate. Corollary~\ref{cor:C7_application} provides a powerful criterion: when $BC^q+B^qD\neq 0$ and $\gcd(a_2,a_0) \neq 1$, we can often construct a variety $\mathcal{C}_0$ containing $\phi$-fixed components that obstruct the APN property. However, computational experiments uncovered exceptional cases where all $\phi$-fixed components of $\mathcal{C}_0$ lie on forbidden hyperplanes. For $q=2$ and $q=4$, we found exactly 16 and 9,120 such functions respectively -- all genuinely APN and all satisfying $C=0$ with specific coefficient relationships.

Our computational searches found 390 and 28,170 APN hexanomials for $\mathbb{F}_{2^2}$ and $\mathbb{F}_{2^4}$ respectively, all belonging to a single CCZ-equivalence class—the known Budaghyan-Carlet family. For larger fields, we found 218 APN functions for $q=8$ in 3 distinct classes and 34 functions for $q=16$ in 2 distinct classes, all outside the BC family. This demonstrates that Dillon's family is significantly richer than previously recognized, validating his 2006 intuition. Notably, none of the discovered APN functions are permutations, suggesting structural incompatibility between this polynomial form and bijectivity.

\subsection{Open questions and future directions}

Several natural questions emerge from our analysis. First, can the threshold $q_0 = 2^{20}$ in Theorem~\ref{Th:Key} be improved? Our computational verification for $q \in \{2,4,8,16\}$ suggests the result holds for all $q \geq 2$, but proving this rigorously would require sharper geometric bounds. The conservative bound arises from worst-case constants in the Cafure-Matera theorem; for the specific varieties in our cases, more refined analysis might yield $q_0 = 2$.

Second, what is the precise algebraic condition forcing $\gcd(a_2,a_0) \neq 1$ when $h_1=0$ and $BC^q+B^qD\neq 0$? While Corollary~\ref{cor:C7_application} handles this case effectively, the complementary situation where $\gcd(a_2,a_0)=1$ remains open theoretically. Our computational experiments show all such instances are non-APN, but understanding why would complete this part of the classification.

Third, can we characterize algebraically exactly when $\gcd(a_2,a_0) \neq 1$ yet all $\phi$-fixed components of $\mathcal{C}_0$ lie on forbidden hyperplanes? The exceptional APN cases we discovered all satisfy $C=0$ with specific coefficient relationships. Understanding this mechanism would transform these computational discoveries into rigorous infinite family constructions. The growth from 16 cases at $q=2$ to 9,120 at $q=4$ suggests the exceptional regime expands substantially with field size.

Looking forward, completing the classification for $q \in \{32,64,128\}$ would definitively identify all APN hexanomials in these fields and verify whether our theoretical obstructions extend to all $q \geq 2$. The discovery of new CCZ-equivalence classes for $q \geq 8$ outside the Budaghyan-Carlet family suggests that larger fields may harbor additional inequivalent constructions.

Our geometric methodology invites generalization to other classes of potential APN functions -- heptanomials, hexanomials with different exponent patterns, or rational functions. More fundamentally, understanding how our obstructions behave under CCZ-equivalence would determine whether we have excluded these functions from being APN in any representation or merely in this specific polynomial form. Alternative geometric tools -- Gr\"obner bases, intersection theory, deformation theory, or \'etale cohomology -- might handle cases our current methods miss or provide improved bounds on $q_0$.

\subsection{Concluding remarks}

Dillon's 2006 conjecture that hexanomials of this form might harbor new APN functions proved prescient. The Budaghyan-Carlet discovery and our computational findings confirm that such functions exist in surprising diversity. However, our systematic analysis reveals they are rare exceptions, emerging only when coefficients  avoid multiple independent obstructions.  

The success of our algebraic-geometric approach exemplifies the power of reformulation in mathematics. By translating combinatorial questions about finite field equations into geometric questions about varieties, we gained access to a rich toolkit -- dimension theory, irreducibility tests, Frobenius actions, point-counting estimates -- that direct computational methods cannot provide. This transformation yielded not only theoretical exclusion results but also guided our computational searches toward promising exceptional regions.
 We have dramatically narrowed the search space and explained why APN hexanomials are rare. Yet we have also identified specific regions where APN functions concentrate, regions that invite further exploration. We hope this technique will prove valuable beyond this specific family, representing a systematic approach applicable to other polynomial families and other problems in finite field theory.

\section*{Acknowledgments}
The third-named author (PS) would like to thank  the first-named author (DB) for the invitation at the Dipartimento di Matematica e Informatica  at Universit\`a degli Studi di Perugia, Italy, and the great working conditions while this paper was being started. The first-named author (DB) and second-named author (GGG) thank the Italian National Group for Algebraic and Geometric Structures and their Applications which supported the research (INdAM -- GNSAGA Project, CUP E53C24001950001).

\appendix


\section*{Summary of computational findings and tables}

Throughout our computational examples, the coefficients $A,B,C,D,E$ and the variable $x$ belong to the field $\mathbb{F}_{q^2}$. The specific constructions for each value of $q$ are as follows:

\paragraph{Field $\mathbb{F}_{4}$ ($q=2$)} For computations where $q=2$, we consider the field $\mathbb{F}_{2^2} = \mathbb{F}_4$. This field is constructed as $\mathbb{F}_2[x]/(x^2+x+1)$. We denote by $\prim$ a primitive element which is a root of the minimal polynomial $x^2+x+1=0$.

\paragraph{Field $\mathbb{F}_{16}$ ($q=4$)} For computations where $q=4$, we consider the field $\mathbb{F}_{4^2} = \mathbb{F}_{16}$. This field is constructed as $\mathbb{F}_2[x]/(x^4+x+1)$. We denote by $\prim$ a primitive element which is a root of the minimal polynomial $x^4+x+1=0$.

\paragraph{Field $\mathbb{F}_{64}$ ($q=8$)} For computations where $q=8$, we consider the field $\mathbb{F}_{8^2} = \mathbb{F}_{64}$. This field is constructed as $\mathbb{F}_2[x]/(x^6+x^4+x^3+x+1)$. We denote by $\prim$ a primitive element which is a root of the minimal polynomial $x^6+x^4+x^3+x+1=0$.

\paragraph{Field $\mathbb{F}_{256}$ ($q=16$)} For computations where $q=16$, we consider the field $\mathbb{F}_{16^2} = \mathbb{F}_{256}$. This field is constructed as $\mathbb{F}_2[x]/(x^8+x^4+x^3+x^2+1)$. We denote by $\prim$ a primitive element which is a root of the minimal polynomial $x^8+x^4+x^3+x^2+1=0$.

We used a SageMath implementation to search for APN functions within the Dillon class. The discovered APN functions were then grouped into CCZ-classes using a Magma implementation available at~\cite{GithubPS25}.

Across all tested fields, \textbf{none} of the discovered APN functions were found to be permutations.

\subsection*{Results on \texorpdfstring{$\mathbb{F}_{2^2}$}{F22} (\texorpdfstring{$q=2$}{q=2})}
An exhaustive search yielded \textbf{390} APN functions that are CCZ-equivalent, summarized in Table~\ref{tab:ccz_F4}.

\begin{longtable}{|c|c|p{7cm}|c|}
\caption{CCZ-Equivalence Class Representatives for $\mathbb{F}_{4}$}
\label{tab:ccz_F4}\\
\hline
\textbf{Class} & \textbf{Count} & \textbf{Rep. Tuple $(A,B,C,D,E)$} & \textbf{Polynomial Representative} \\
\hline
\endfirsthead

\multicolumn{4}{c}
{{\tablename\ \thetable{} -- continued from previous page}} \\
\hline
\textbf{Class} & \textbf{Count} & \textbf{Rep. Tuple $(A,B,C,D,E)$} & \textbf{Polynomial Representative} \\
\hline
\endhead

\hline \multicolumn{4}{|r|}{{Continued on next page}} \\
\endfoot

\hline
\endlastfoot

1 & 390 & $(\prim, 0, 0, 0, \prim)$ & $\prim^{2} x^{6} + \prim x^{3}$ \\
\hline
\end{longtable}

\subsection*{Results on \texorpdfstring{$\mathbb{F}_{2^4}$}{F24} (\texorpdfstring{$q=4$}{q=4})}
An exhaustive search yielded \textbf{28,170} APN functions, all of which are  CCZ-equivalent, summarized in  Table~\ref{tab:ccz_F16}.

\begin{longtable}{|c|c|p{7cm}|c|}
\caption{CCZ-Equivalence Class Representatives for $\mathbb{F}_{16}$}
\label{tab:ccz_F16}\\
\hline
\textbf{Class} & \textbf{Count} & \textbf{Rep. Tuple $(A,B,C,D,E)$} & \textbf{Polynomial Representative} \\
\hline
\endfirsthead

\multicolumn{4}{c}
{{\tablename\ \thetable{} -- continued from previous page}} \\
\hline
\textbf{Class} & \textbf{Count} & \textbf{Rep. Tuple $(A,B,C,D,E)$} & \textbf{Polynomial Representative} \\
\hline
\endhead

\hline \multicolumn{4}{|r|}{{Continued on next page}} \\
\endfoot

\hline
\endlastfoot

1 & 28170 & $(\prim, 0, 0, \prim, 0)$ & $x^{12} + \prim x^{6} + \prim x^{3}$ \\
\hline
\end{longtable}

\subsection*{Results on $\mathbb{F}_{2^6}$ ($q=8$)}
A random search of  120,000 candidate tuples found \textbf{218} APN functions. These belong to exactly 3 distinct CCZ-equivalent classes, summarized in Table~\ref{tab:ccz_F64}.  

{\footnotesize
\begin{longtable}{|c|c|p{7cm}|c|}
\caption{CCZ-Equivalence Class Representatives for $\mathbb{F}_{64}$}
\label{tab:ccz_F64}\\
\hline
\textbf{Class} & \textbf{Count} & \textbf{Rep. Tuple $(A,B,C,D,E)$} & \textbf{Polynomial Representative} \\
\hline
\endfirsthead

\multicolumn{4}{c}
{{\tablename\ \thetable{} -- continued from previous page}} \\
\hline
\textbf{Class} & \textbf{Count} & \textbf{Rep. Tuple $(A,B,C,D,E)$} & \textbf{Polynomial Representative} \\
\hline
\endhead

\hline \multicolumn{4}{|r|}{{Continued on next page}} \\
\endfoot

\hline
\endlastfoot

1 & 83 & $(\prim^{23}, \prim^{23}, \prim^{47}, \prim^{25}, \prim^{29})$ & $x^{24} + \prim^{29} x^{18} + \prim^{47} x^{17} + \prim^{25} x^{10} + \prim^{23} x^{9} + \prim^{23} x^{3}$ \\
\hline
2 & 133 & $(\prim^{35}, \prim^{46}, \prim^{6}, \prim^{20}, \prim^{31})$ & $x^{24} + \prim^{31} x^{18} + \prim^{6} x^{17} + \prim^{20} x^{10} + \prim^{46} x^{9} + \prim^{35} x^{3}$ \\
\hline
3 & 2 & $(\prim^{37}, 0, \prim^{41}, \prim^{28}, 0)$ & $x^{24} + \prim^{41} x^{17} + \prim^{28} x^{10} + \prim^{37} x^{3}$ \\
\hline
\end{longtable}
}

\subsection*{Results on $\mathbb{F}_{2^8}$ ($q=16$)}
A random search of  200,000 candidate tuples yielded \textbf{34} APN functions. A complete pairwise CCZ-equivalence check was performed on these functions, and they were grouped into exactly  2 distinct CCZ-equivalent classes, as shown in Table~\ref{tab:ccz_F256}.

{\footnotesize
\begin{longtable}{|c|c|p{7cm}|c|}
\caption{CCZ-Equivalence Class Representatives for $\mathbb{F}_{256}$}
\label{tab:ccz_F256}\\
\hline
\textbf{Class} & \textbf{Count} & \textbf{Rep. Tuple $(A,B,C,D,E)$} & \textbf{Polynomial Representative} \\
\hline
\endfirsthead

\multicolumn{4}{c}
{{\tablename\ \thetable{} -- continued from previous page}} \\
\hline
\textbf{Class} & \textbf{Count} & \textbf{Rep. Tuple $(A,B,C,D,E)$} & \textbf{Polynomial Representative} \\
\hline
\endhead

\hline \multicolumn{4}{|r|}{{Continued on next page}} \\
\endfoot

\hline
\endlastfoot

1 & 28 & $(\prim^{210}, \prim^{34}, \prim^{125}, \prim^{170}, \prim^{207})$ & $\prim^{85} x^3 + \prim^{119} x^{17} + \prim^{17} x^{33} + \prim^{187} x^{18} + \prim^{170} x^{34} + x^{48}$ \\
\hline
2 & 6 & $(\prim^{25}, \prim^{51}, \prim^{34}, \prim^{68}, \prim^{17})$ & $\prim^{25} x^3 + \prim^{51} x^{17} + \prim^{34} x^{33} + \prim^{68} x^{18} + \prim^{17} x^{34} + x^{48}$ \\
\hline
\end{longtable}
 }


\textbf{Tables mentioned in Remark~\ref{rem6.4}}

\begin{table}[H]
\small
\centering
\caption{APN Functions satisfying $h_1=0$ and $BC^q+B^qD\neq 0$ for $q=2$}
\label{tab:apn_q2}
\begin{tabular}{cl|cl}
\hline
\textbf{\#} & \textbf{Simplified APN Polynomial} & \textbf{\#} & \textbf{Simplified APN Polynomial} \\
\hline
1 & $(a+1)x^4 + ax^5 + ax^6$ & 15 & $(a+1)x^3 + ax^5 + x^6$ \\
2 & $x^4 + x^5 + ax^6$ & 16 & $(a+1)x^3 + (a+1)x^5 + (a+1)x^6$ \\
3 & $(a+1)x^3 + ax^4 + ax^5 + ax^6$ & 17 & $(a+1)x^3 + (a+1)x^5 + x^6$ \\
4 & $(a+1)x^3 + x^4 + (a+1)x^5 + ax^6$ & 18 & $(a+1)x^3 + x^5 + x^6$ \\
5 & $ax^4 + (a+1)x^5 + (a+1)x^6$ & 19 & $ax^3 + ax^4 + (a+1)x^6$ \\
6 & $x^4 + x^5 + (a+1)x^6$ & 20 & $ax^3 + ax^4 + x^6$ \\
7 & $ax^3 + x^4 + ax^5 + (a+1)x^6$ & 21 & $ax^3 + (a+1)x^4 + (a+1)x^6$ \\
8 & $ax^3 + (a+1)x^4 + (a+1)x^5 + (a+1)x^6$ & 22 & $ax^3 + (a+1)x^4 + x^6$ \\
9 & $(a+1)x^3 + ax^4 + ax^6$ & 23 & $ax^3 + x^4$ \\
10 & $(a+1)x^3 + ax^4$ & 24 & $ax^3 + ax^5 + (a+1)x^6$ \\
11 & $(a+1)x^3 + (a+1)x^4 + ax^6$ & 25 & $ax^3 + ax^5 + x^6$ \\
12 & $(a+1)x^3 + (a+1)x^4$ & 26 & $ax^3 + (a+1)x^5 + (a+1)x^6$ \\
13 & $(a+1)x^3 + x^4$ & 27 & $ax^3 + (a+1)x^5 + x^6$ \\
14 & $(a+1)x^3 + ax^5 + ax^6$ & 28 & $ax^3 + x^5$ \\
\hline
\end{tabular}
\end{table}

\begin{longtable}{|c|p{0.8\textwidth}|}

\caption{APN Functions satisfying $h_1=0$, $BC^q+B^qD\neq 0$, 
and the exceptional condition $\gcd(a_2,a_0)\neq 1$ with 
$\mathcal{C}_0 \subseteq \pi_1 \cup \pi_2$ for $q=4$}
\label{tab:apn_q4} \\

\hline
\textbf{\#} & \textbf{Polynomial} \\
\hline
\endhead

\hline
\multicolumn{2}{|r|}{{Continued on next page}} \\
\endfoot

\hline
\endlastfoot

1 & $ax^{3} + ax^{5} + (a^2 + a + 1)x^{6} + a^2x^{10} + x^{12}$ \\
2 & $ax^{3} + ax^{5} + (a^2 + a + 1)x^{6} + (a^3 + a^2 + 1)x^{10} + x^{12}$ \\
3 & $ax^{3} + ax^{5} + (a^2 + a + 1)x^{6} + (a^3 + 1)x^{10} + x^{12}$ \\
4 & $ax^{3} + ax^{5} + ax^{9} + (a^2 + 1)x^{6} + (a + 1)x^{10} + x^{12}$ \\
5 & $ax^{3} + ax^{5} + ax^{9} + (a^2 + 1)x^{6} + (a^2 + a)x^{10} + x^{12}$ \\
6 & $ax^{3} + ax^{5} + ax^{9} + (a^2 + 1)x^{6} + (a^2 + 1)x^{10} + x^{12}$ \\
7 & $ax^{3} + ax^{5} + a^2x^{9} + x^{6} + a^3x^{10} + x^{12}$ \\
8 & $ax^{3} + ax^{5} + a^2x^{9} + x^{6} + (a^2 + a + 1)x^{10} + x^{12}$ \\
9 & $ax^{3} + ax^{5} + a^2x^{9} + x^{6} + (a^3 + a^2 + a + 1)x^{10} + x^{12}$ \\
10 & $ax^{3} + ax^{5} + (a^2 + a)x^{9} + (a + 1)x^{6} + a^3x^{10} + x^{12}$ \\
11 & $ax^{3} + ax^{5} + (a^2 + a)x^{9} + (a + 1)x^{6} + (a^3 + a + 1)x^{10} + x^{12}$ \\
12 & $ax^{3} + ax^{5} + (a^2 + a)x^{9} + (a + 1)x^{6} + (a^3 + a)x^{10} + x^{12}$ \\
13 & $ax^{3} + ax^{5} + (a^2 + a)x^{9} + (a + 1)x^{6} + (a^2 + a + 1)x^{10} + x^{12}$ \\
14 & $ax^{3} + ax^{5} + (a^2 + a)x^{9} + (a + 1)x^{6} + (a^3 + a^2 + a + 1)x^{10} + x^{12}$ \\
15 & $ax^{3} + ax^{5} + (a^2 + a)x^{9} + (a + 1)x^{6} + x^{10} + x^{12}$ \\
16 & $ax^{3} + ax^{5} + (a^3 + a + 1)x^{9} + a^3x^{10} + x^{12}$ \\
17 & $ax^{3} + ax^{5} + (a^3 + a + 1)x^{9} + (a + 1)x^{10} + x^{12}$ \\
18 & $ax^{3} + ax^{5} + (a^3 + a + 1)x^{9} + (a^2 + 1)x^{10} + x^{12}$ \\
19 & $ax^{3} + ax^{5} + (a^3 + a + 1)x^{9} + (a^3 + a^2 + a + 1)x^{10} + x^{12}$ \\
20 & $ax^{3} + ax^{5} + (a^2 + 1)x^{9} + (a^3 + a^2 + a)x^{6} + a^3x^{10} + x^{12}$ \\
21 & $ax^{3} + ax^{5} + (a^2 + 1)x^{9} + (a^3 + a^2 + a)x^{6} + (a^2 + a + 1)x^{10} + x^{12}$ \\
22 & $ax^{3} + ax^{5} + (a^2 + 1)x^{9} + (a^3 + a^2 + a)x^{6} + (a^3 + a^2 + a + 1)x^{10} + x^{12}$ \\
23 & $ax^{3} + ax^{5} + (a^3 + a)x^{9} + (a^3 + a^2 + a + 1)x^{6} + a^2x^{10} + x^{12}$ \\
24 & $ax^{3} + ax^{5} + (a^3 + a)x^{9} + (a^3 + a^2 + a + 1)x^{6} + (a^3 + a^2 + 1)x^{10} + x^{12}$ \\
25 & $ax^{3} + ax^{5} + (a^3 + a)x^{9} + (a^3 + a^2 + a + 1)x^{6} + (a^3 + 1)x^{10} + x^{12}$ \\
26 & $ax^{3} + ax^{5} + (a^3 + a^2 + a)x^{9} + (a^3 + 1)x^{6} + (a + 1)x^{10} + x^{12}$ \\
27 & $ax^{3} + ax^{5} + (a^3 + a^2 + a)x^{9} + (a^3 + 1)x^{6} + (a^2 + a)x^{10} + x^{12}$ \\
28 & $ax^{3} + ax^{5} + (a^3 + a^2 + a)x^{9} + (a^3 + 1)x^{6} + (a^2 + 1)x^{10} + x^{12}$ \\
29 & $ax^{3} + ax^{5} + (a^3 + a^2 + a + 1)x^{9} + (a^2 + a)x^{6} + a^2x^{10} + x^{12}$ \\
30 & $ax^{3} + ax^{5} + (a^3 + a^2 + a + 1)x^{9} + (a^2 + a)x^{6} + (a^3 + a + 1)x^{10} + x^{12}$ \\
31 & $ax^{3} + ax^{5} + (a^3 + a^2 + a + 1)x^{9} + (a^2 + a)x^{6} + (a^3 + a)x^{10} + x^{12}$ \\
32 & $ax^{3} + ax^{5} + (a^3 + a^2 + a + 1)x^{9} + (a^2 + a)x^{6} + (a^3 + a^2 + 1)x^{10} + x^{12}$ \\
33 & $ax^{3} + ax^{5} + (a^3 + a^2 + a + 1)x^{9} + (a^2 + a)x^{6} + (a^3 + 1)x^{10} + x^{12}$ \\
34 & $ax^{3} + ax^{5} + (a^3 + a^2 + a + 1)x^{9} + (a^2 + a)x^{6} + x^{10} + x^{12}$ \\
35 & $ax^{3} + ax^{5} + (a^3 + a^2 + 1)x^{9} + a^2x^{6} + a^2x^{10} + x^{12}$ \\
36 & $ax^{3} + ax^{5} + (a^3 + a^2 + 1)x^{9} + a^2x^{6} + (a + 1)x^{10} + x^{12}$ \\
37 & $ax^{3} + ax^{5} + (a^3 + a^2 + 1)x^{9} + a^2x^{6} + (a^2 + a)x^{10} + x^{12}$ \\
38 & $ax^{3} + ax^{5} + (a^3 + a^2 + 1)x^{9} + a^2x^{6} + (a^2 + 1)x^{10} + x^{12}$ \\
39 & $ax^{3} + ax^{5} + (a^3 + a^2 + 1)x^{9} + a^2x^{6} + (a^3 + a^2 + 1)x^{10} + x^{12}$ \\
40 & $ax^{3} + ax^{5} + (a^3 + a^2 + 1)x^{9} + a^2x^{6} + (a^3 + 1)x^{10} + x^{12}$ \\
41 & $ax^{3} + ax^{5} + (a^3 + 1)x^{9} + ax^{6} + (a^3 + a + 1)x^{10} + x^{12}$ \\
42 & $ax^{3} + ax^{5} + (a^3 + 1)x^{9} + ax^{6} + (a^3 + a)x^{10} + x^{12}$ \\
43 & $ax^{3} + ax^{5} + (a^3 + 1)x^{9} + ax^{6} + x^{10} + x^{12}$ \\
44 & $ax^{3} + ax^{5} + x^{9} + a^3x^{6} + (a^3 + a + 1)x^{10} + x^{12}$ \\
45 & $ax^{3} + ax^{5} + x^{9} + a^3x^{6} + (a^3 + a)x^{10} + x^{12}$ \\
46 & $ax^{3} + ax^{5} + x^{9} + a^3x^{6} + x^{10} + x^{12}$ \\
47 & $ax^{3} + a^2x^{5} + (a^3 + a)x^{6} + ax^{10} + x^{12}$ \\
48 & $ax^{3} + a^2x^{5} + (a^3 + a)x^{6} + a^3x^{10} + x^{12}$ \\
44 & $ax^{3} + a^2x^{5} + (a^3 + a)x^{6} + (a + 1)x^{10} + x^{12}$ \\
50 & $ax^{3} + a^2x^{5} + (a^3 + a)x^{6} + (a^2 + a)x^{10} + x^{12}$ \\
\hline
\multicolumn{2}{|c|}{\textbf{... Rows 51-9070 Omitted (Total 9120 entries) ...}} \\
\hline
9071 & $x^{3} + (a^3 + a + 1)x^{5} + a^3x^{9} + (a^2 + a)x^{6} + (a^3 + a^2)x^{10} + x^{12}$ \\
9072 & $x^{3} + (a^3 + a + 1)x^{5} + a^3x^{9} + (a^2 + a)x^{6} + (a^3 + a^2 + 1)x^{10} + x^{12}$ \\
9073 & $x^{3} + (a^3 + a + 1)x^{5} + a^3x^{9} + (a^2 + a)x^{6} + (a^3 + a)x^{10} + x^{12}$ \\
9074 & $x^{3} + (a^3 + a + 1)x^{5} + a^3x^{9} + (a^2 + a)x^{6} + (a + 1)x^{10} + x^{12}$ \\
9075 & $x^{3} + (a^3 + a + 1)x^{5} + a^3x^{9} + (a^2 + a)x^{6} + (a^2 + a)x^{10} + x^{12}$ \\
9076 & $x^{3} + (a^3 + a + 1)x^{5} + a^3x^{9} + (a^2 + a)x^{6} + (a^2 + 1)x^{10} + x^{12}$ \\
9077 & $x^{3} + (a^3 + a + 1)x^{5} + (a^3 + a^2)x^{9} + (a^3 + a + 1)x^{6} + (a^3 + a^2)x^{10} + x^{12}$ \\
9078 & $x^{3} + (a^3 + a + 1)x^{5} + (a^3 + a^2)x^{9} + (a^3 + a + 1)x^{6} + (a + 1)x^{10} + x^{12}$ \\
9079 & $x^{3} + (a^3 + a + 1)x^{5} + (a^3 + a^2)x^{9} + (a^3 + a + 1)x^{6} + (a^2 + a)x^{10} + x^{12}$ \\
9080 & $x^{3} + (a^3 + a + 1)x^{5} + (a^3 + a^2)x^{9} + (a^3 + a + 1)x^{6} + (a^3 + a)x^{10} + x^{12}$ \\
9081 & $x^{3} + (a^3 + a + 1)x^{5} + (a^3 + a^2 + a)x^{9} + (a^2 + a + 1)x^{6} + (a^2 + 1)x^{10} + x^{12}$ \\
9082 & $x^{3} + (a^3 + a + 1)x^{5} + (a^3 + a^2 + a)x^{9} + (a^2 + a + 1)x^{6} + (a^3 + a^2 + a)x^{10} + x^{12}$ \\
9083 & $x^{3} + (a^3 + a + 1)x^{5} + (a^3 + a^2 + a)x^{9} + (a^2 + a + 1)x^{6} + (a^3 + a)x^{10} + x^{12}$ \\
9084 & $x^{3} + (a^3 + a + 1)x^{5} + (a^3 + a^2 + a)x^{9} + (a^2 + a + 1)x^{6} + (a^3 + a^2 + 1)x^{10} + x^{12}$ \\
9085 & $x^{3} + (a^3 + a + 1)x^{5} + (a^3 + a^2 + a + 1)x^{9} + a^3x^{6} + ax^{10} + x^{12}$ \\
9086 & $x^{3} + (a^3 + a + 1)x^{5} + (a^3 + a^2 + a + 1)x^{9} + a^3x^{6} + a^3x^{10} + x^{12}$ \\
9087 & $x^{3} + (a^3 + a + 1)x^{5} + (a^3 + a^2 + a + 1)x^{9} + a^3x^{6} + (a + 1)x^{10} + x^{12}$ \\
9088 & $x^{3} + (a^3 + a + 1)x^{5} + (a^3 + a^2 + a + 1)x^{9} + a^3x^{6} + (a^2 + a)x^{10} + x^{12}$ \\
9089 & $x^{3} + (a^3 + a + 1)x^{5} + (a^3 + a^2 + a + 1)x^{9} + a^3x^{6} + (a^2 + 1)x^{10} + x^{12}$ \\
9090 & $x^{3} + (a^3 + a + 1)x^{5} + (a^3 + a^2 + a + 1)x^{9} + a^3x^{6} + (a^3 + a^2)x^{10} + x^{12}$ \\
9091 & $x^{3} + (a^3 + a + 1)x^{5} + (a^3 + a^2 + a + 1)x^{9} + a^3x^{6} + (a^3 + a^2 + a)x^{10} + x^{12}$ \\
9092 & $x^{3} + (a^3 + a + 1)x^{5} + (a^3 + a^2 + a + 1)x^{9} + a^3x^{6} + (a^3 + a + 1)x^{10} + x^{12}$ \\
9093 & $x^{3} + (a^3 + a + 1)x^{5} + (a^3 + a^2 + a + 1)x^{9} + a^3x^{6} + (a^3 + a)x^{10} + x^{12}$ \\
9094 & $x^{3} + (a^3 + a + 1)x^{5} + (a^3 + a^2 + a + 1)x^{9} + a^3x^{6} + x^{10} + x^{12}$ \\
9095 & $x^{3} + (a^3 + a + 1)x^{5} + (a^3 + a^2 + 1)x^{9} + (a^3 + a^2 + a)x^{6} + (a^3 + a^2)x^{10} + x^{12}$ \\
9096 & $x^{3} + (a^3 + a + 1)x^{5} + (a^3 + a^2 + 1)x^{9} + (a^3 + a^2 + a)x^{6} + (a + 1)x^{10} + x^{12}$ \\
9097 & $x^{3} + (a^3 + a + 1)x^{5} + (a^3 + a^2 + 1)x^{9} + (a^3 + a^2 + a)x^{6} + (a^2 + a)x^{10} + x^{12}$ \\
9098 & $x^{3} + (a^3 + a + 1)x^{5} + (a^3 + a^2 + 1)x^{9} + (a^3 + a^2 + a)x^{6} + (a^2 + 1)x^{10} + x^{12}$ \\
9099 & $x^{3} + (a^3 + a + 1)x^{5} + (a^3 + a^2 + 1)x^{9} + (a^3 + a^2 + a)x^{6} + (a^3 + a)x^{10} + x^{12}$ \\
9100 & $x^{3} + (a^3 + a + 1)x^{5} + (a^3 + a^2 + 1)x^{9} + (a^3 + a^2 + a)x^{6} + (a^3 + 1)x^{10} + x^{12}$ \\
9101 & $x^{3} + (a^3 + a + 1)x^{5} + (a^3 + 1)x^{9} + (a^2 + 1)x^{6} + (a^3 + a + 1)x^{10} + x^{12}$ \\
9102 & $x^{3} + (a^3 + a + 1)x^{5} + (a^3 + 1)x^{9} + (a^2 + 1)x^{6} + (a^3 + a)x^{10} + x^{12}$ \\
9103 & $x^{3} + (a^3 + a + 1)x^{5} + (a^3 + 1)x^{9} + (a^2 + 1)x^{6} + x^{10} + x^{12}$ \\
9104 & $x^{3} + (a^3 + a + 1)x^{5} + x^{9} + (a^3 + a^2)x^{6} + ax^{10} + x^{12}$ \\
9105 & $x^{3} + (a^3 + a + 1)x^{5} + x^{9} + (a^3 + a^2)x^{6} + a^3x^{10} + x^{12}$ \\
9106 & $x^{3} + (a^3 + a + 1)x^{5} + x^{9} + (a^3 + a^2)x^{6} + (a^2 + a + 1)x^{10} + x^{12}$ \\
9107 & $x^{3} + (a^3 + a + 1)x^{5} + x^{9} + (a^3 + a^2)x^{6} + (a^3 + a^2 + a)x^{10} + x^{12}$ \\
9108 & $x^{3} + (a^3 + a + 1)x^{5} + x^{9} + (a^3 + a^2)x^{6} + (a^3 + 1)x^{10} + x^{12}$ \\
9109 & $x^{3} + (a^3 + a + 1)x^{5} + x^{9} + (a^3 + a^2)x^{6} + x^{10} + x^{12}$ \\
9110 & $x^{3} + (a^3 + 1)x^{5} + (a^3 + a^2 + a)x^{6} + (a^2 + a + 1)x^{10} + x^{12}$ \\
9111 & $x^{3} + (a^3 + 1)x^{5} + (a^3 + a^2 + a)x^{6} + (a^3 + a^2 + a)x^{10} + x^{12}$ \\
9112 & $x^{3} + (a^3 + 1)x^{5} + (a^3 + a^2 + a)x^{6} + (a^3 + a)x^{10} + x^{12}$ \\
9113 & $x^{3} + (a^3 + 1)x^{5} + (a^3 + a^2 + a + 1)x^{9} + (a^2 + 1)x^{6} + (a^2 + a)x^{10} + x^{12}$ \\
9114 & $x^{3} + (a^3 + 1)x^{5} + (a^3 + a^2 + a + 1)x^{9} + (a^2 + 1)x^{6} + (a^3 + a + 1)x^{10} + x^{12}$ \\
9115 & $x^{3} + (a^3 + 1)x^{5} + (a^3 + a^2 + a + 1)x^{9} + (a^2 + 1)x^{6} + (a^3 + a^2 + 1)x^{10} + x^{12}$ \\
9116 & $x^{3} + (a^3 + 1)x^{5} + (a^3 + a^2 + 1)x^{9} + ax^{6} + (a^2 + a + 1)x^{10} + x^{12}$ \\
9117 & $x^{3} + (a^3 + 1)x^{5} + (a^3 + a^2 + 1)x^{9} + ax^{6} + (a^3 + a^2 + a)x^{10} + x^{12}$ \\
9118 & $x^{3} + (a^3 + 1)x^{5} + (a^3 + a^2 + 1)x^{9} + ax^{6} + (a^3 + a^2 + 1)x^{10} + x^{12}$ \\
9119 & $x^{3} + (a^3 + 1)x^{5} + (a^3 + 1)x^{9} + (a^3 + a^2 + a + 1)x^{6} + (a^3 + a + 1)x^{10} + x^{12}$ \\
9120 & $x^{3} + (a^3 + 1)x^{5} + (a^3 + 1)x^{9} + (a^3 + a^2 + a + 1)x^{6} + (a^3 + a)x^{10} + x^{12}$ \\
\end{longtable}

\end{document}